\documentclass[reqno]{amsart}

\usepackage{amscd,amsthm,amsmath,amssymb,mathtools}
\usepackage{upgreek,bm}
\usepackage{verbatim}
\usepackage{color}
\usepackage{url}
\usepackage{enumerate,enumitem}
\usepackage{graphicx}

\usepackage{algorithm}
\usepackage{algorithmic}
\usepackage{epstopdf,epsfig,subfigure}
\usepackage{curve2e}
\usepackage{mathrsfs}
\usepackage{array}
\usepackage{stackrel}

\usepackage[numbers,sort&compress]{natbib}

\usepackage{setspace}
\usepackage[
    bookmarks=true,         
    unicode=false,          
    pdftoolbar=true,        
    pdfmenubar=true,        
    pdffitwindow=false,     
    pdfstartview={FitH},    
    pdftitle={Output-based receding-horizon stabilizing control},    
    pdfauthor={Author},     
    pdfsubject={Subject},   
    pdfcreator={Creator},   
    pdfproducer={Producer}, 
    pdfkeywords={keyword1, key2, key3}, 
    pdfnewwindow=true,      
    colorlinks=true,       
    linkcolor=red,          
    citecolor=blue,        
    filecolor=magenta,      
    urlcolor=cyan,           
hypertexnames=false
]{hyperref}

\theoremstyle{plain}

\newtheorem{theorem}{Theorem}[section]

\newtheorem{lemma}[theorem]{Lemma}

\newtheorem{problem}[theorem]{Problem}
\newtheorem{assumption}[theorem]{Assumption}

\theoremstyle{definition}
\newtheorem{definition}[theorem]{Definition}
\newtheorem{remark}[theorem]{Remark}

\numberwithin{equation}{section}

\usepackage{geometry}

\newcommand{\linspan}{\mathop{\rm span}\nolimits}

\newcommand{\rest}{\left.\kern-2\nulldelimiterspace\right|_}
\newcommand{\norm}[2]{\left|#1\right|_{#2}}

\newcommand{\Id}{{\mathbf1}}
\newcommand{\indf}{1}

\newcommand{\ex}{\mathrm{e}}
\newcommand{\p}{\partial}

\newcommand*{\Bigcdot}{\raisebox{-.25ex}{\scalebox{1.25}{$\cdot$}}}

\newcommand{\clE}{{\mathcal E}}

\newcommand{\clH}{{\mathcal H}}
\newcommand{\clI}{{\mathcal I}}
\newcommand{\clJ}{{\mathcal J}}

\newcommand{\clL}{{\mathcal L}}

\newcommand{\clT}{{\mathcal T}}
\newcommand{\clU}{{\mathcal U}}
\newcommand{\clV}{{\mathcal V}}
\newcommand{\clW}{{\mathcal W}}
\newcommand{\clX}{{\mathcal X}}

\newcommand{\clZ}{{\mathcal Z}}

\newcommand{\bbN}{{\mathbb N}}

\newcommand{\bbR}{{\mathbb R}}

\newcommand{\bfZ}{{\mathbf Z}}

\newcommand{\fkA}{{\mathfrak A}}
\newcommand{\fkB}{{\mathfrak B}}
\newcommand{\fkC}{{\mathfrak C}}

\newcommand{\fkI}{{\mathfrak I}}
\newcommand{\fkJ}{{\mathfrak J}}
\newcommand{\fkK}{{\mathfrak K}}
\newcommand{\fkL}{{\mathfrak L}}
\newcommand{\fkM}{{\mathfrak M}}

\newcommand{\fkT}{{\mathfrak T}}
\newcommand{\fkU}{{\mathfrak U}}

\newcommand{\fkW}{{\mathfrak W}}

\newcommand{\rmD}{{\mathrm D}}

\newcommand{\bfn}{{\mathbf n}}

\newcommand{\bfz}{{\mathbf z}}

\newcommand{\rmd}{{\mathrm d}}
\newcommand{\rme}{{\mathrm e}}
\newcommand{\rmf}{{\mathrm f}}

\newcommand{\rmo}{{\mathrm o}}

\newcommand{\ovlineC}[1]{\overline C_{\left[#1\right]}}

\definecolor{DarkBlue}{rgb}{0,0.08,0.45}
\definecolor{DarkRed}{rgb}{.65,0,0}
\definecolor{applegreen}{rgb}{0.55, 0.71, 0.0}

\newcounter{mymac@matlab}
\newcommand{\matlab}{MATLAB%
   \ifnum\value{mymac@matlab}<1%
   \textregistered%
   \setcounter{mymac@matlab}{1}%
   \fi%
  }

\providecommand{\argmin}{\operatorname*{argmin}}

\pdfoutput=1

\begin{document}
\title{Output-based receding horizon stabilizing control}
\author{Behzad Azmi$^{\tt1}$}
\author{S\'ergio S.~Rodrigues$^{\tt2*}$}
\thanks{
\vspace{-1em}\newline\noindent
\newline\noindent
$^{\tt1}$ Dep.  Math. Stat., Univ.
Konstanz, Universitätstr. 10, D-78457 Konstanz, Germany.
  \newline\noindent
$^{\tt2}$ Johann Radon Inst. Comput. Appl. Math.,
  \"OAW, Altenbergerstr. 69, 4040 Linz, Austria.
\quad
\newline\noindent
{\sc Emails}:
{\small\tt behzad.azmi@uni-konstanz.de,\quad sergio.rodrigues@ricam.oeaw.ac.at}%
}

\begin{abstract}
A receding horizon control framework is coupled with a Luenberger observer to construct an output-based control input stabilizing parabolic equations. The actuators and sensors are indicator functions of small subdomains, representing localized actuation and localized measurements. It is shown that, for a class of explicitly given sets of actuators and sensors, we can guarantee the stabilizing property of the constructed input.  Results of numerical simulations are presented validating the theoretical findings.
\end{abstract}

\maketitle

\pagestyle{myheadings} \thispagestyle{plain} \markboth{\sc  B. Azmi and S. S.
Rodrigues}{\sc output-based receding horizon control}

\section{Introduction}
We consider evolutionary linear parabolic-like equations for time ~$t\ge0$ given by
 \begin{align}\label{sys-y-intro0} 
 \dot y +Ay+A_{\rm rc}y =Bu,\qquad y(0)= y_0,\qquad w=\clZ y,
\end{align} 
where  the state ~$y(t)$ evolves in a real Hilbert space,  influenced by a finite-dimensional control input ~$u(t)\in\bbR^M$.  This input is used to tune $M$ available actuators.

Our primary interest is in scenarios where the free dynamics (i.e., when 
~$u(t)\in\bbR^M$) may be unstable,  and we aim to stabilize them.  Additionally, we seek to minimize the energy spent during the stabilization process.

At time~$t\ge0$,  we have only access to partial information on the state~$y(t)$ through the output~$w(t)=\clZ y(t)\in\bbR^S$,  which is provided by a finite number ~$S$ of sensors.  Notably,  the initial state ~$y_0\coloneqq y(0)$ at initial time~$t=0$ is not completely known.

Our goal is to determine the control input in the feedback form ~$u(t)=K(t,y(t))$,  which depends on  time~$t$ and state~$y(t)$ at time~$t$.  However, since ~$y(t)$ is not fully known,  we will instead seek an output-based feedback control input~$u(t)=K(t,\widehat y(t))$,  where~$\widehat y(t)$ is an estimate for~$y(t)$.  This estimate will be obtained by an exponential Luenberger observer,  constructed using the available output ~$w(t)$.

\subsection{Main contributions}
The main result of this manuscript concerns the stability of the closed-loop system that couples a receding horizon control (RHC) with a dynamical Luenberger observer. This observer provides an estimate for the state of the system, based on sensor measurements. Despite the crucial importance of such closed-loop coupled systems for various applications,  we could not find such stability results in the RHC-related literature, neither in the context of models governed by parabolic equations,  as addressed in this manuscript,  nor in the more general context of {\em continuous-time} dynamical systems. 

The theoretical findings are supported by numerical experiments, demonstrating the stabilizing performance of the proposed output-based receding horizon control (ORHC) framework.

\subsection{Some related literature}
RHC,  also known as model predictive control (MPC), is an efficient framework for addressing infinite horizon optimal control problems,  see~\cite{GruenePannek17,RawlingsMayneDiehl19}, for example.   In this framework, the solution to an infinite horizon problem is approximated by a sequence of finite horizon problems in a receding horizon manner. Ensuring stabilizability and optimality with finite prediction horizons remains a significant challenge.

For autonomous linear unconstrained dynamics, a stabilizing control can be obtained by solving a suitable algebraic Riccati equation; see~\cite[Part~III, Sect.~1.3; Part~IV, Sect.~4.4]{Zabczyk08} and~\cite[Sect.~7.1.4]{Bacciotti19}.  In the case of time-periodic dynamics, the solution can be found by solving a time-periodic differential Riccati equation see~\cite{Lunardi91,Rod23-eect}. However, for general nonautonomous dynamics, solving the associated differential Riccati equations is not feasible (cf.~\cite[Sect.~2.3]{Rod23-eect}). In such cases, a RHC framework can still be applied, as we will demonstrate in this manuscript.

RHC has gained popularity over the past few decades due to its flexibility in handling constraints, nonautonomous dynamics, and nonlinear dynamics. This approach has been widely applied in both continuous-time dynamical systems \cite{Longchamp83,MayneMichalska90,LiYanShi17,KwonPearson77,KunPfeiffer20,AzmiKunisch19,VeldZua22,VeldBorZua24} and discrete-time dynamical systems \cite{MayneRawlRaoScok00,GruenePannek09,GrueneRantzer08,FranzeLucia15,GarciaPrettMorari89}. Regarding the integration of RHC with state estimators, \cite{MayneRavicFindAllg09} explored the use of a Luenberger state estimator with RHC for {\em  finite-dimensional discrete-time} controlled systems and the stabilizability of ORHC was ensured by using a robustly stabilizing tube-based MPC, involving terminal costs and constraints, to control the state of the observer.  Similarly, \cite{ItoKuni06} combined RHC with a linear quadratic Gaussian state estimator for finite-dimensional continuous-time dynamical systems,  guaranteeing stabilizability through using terminal costs in the finite time-horizon problems.

In the present work, we investigate  {\em continuous-time infinite-dimensional} controlled systems governed by parabolic equations. Our  ORHC framework does not require terminal costs or constraints; instead,  stabilizability is achieved through overlapping intervals and a suitable concatenation scheme.

\subsection{Contents}
The rest of the paper is organized as follows: In Section~\ref{S:orhc},  the abstract problem formulation, the cost functional, and the ORHC  algorithm are presented.  Section~\ref{S:stabilizability} first introduces the required assumptions on the operators~$A$, $A_{\rm rc}$, $B$, and~$\clZ$.  Then, relying on these assumptions and on a further assumption on the cost functional, the main result on  the exponential stabilizability of ORHC is proved.  Further, in Section~\ref{S:illust-parabolic},  the applicability of the theoretical results from the previous section is investigated for a class of nonautonomous linear parabolic equations with finitely many indicator function-based actuators and sensors. Section~\ref{S:simul} reports on numerical simulations validating our theoretical findings. Finally, Section~\ref{S:finalremks} gathers concluding remarks on the achieved results and discusses interesting subjects for potential follow-up studies.

\section{Output based receding horizon control}\label{S:orhc}
We present the details concerning the output based receding-horizon control and the strategy to prove the theoretical result.

\subsection{Stabilization while minimizing the spent energy}
We consider evolutionary linear parabolic-like equations, for time~$t\ge0$, with free dynamics given by
 \begin{align}\label{sys-y-intro} 
 \dot y +Ay+A_{\rm rc}y =0,\qquad y(0)= y_0,
\end{align} 
with state~$y(t)\in H$ evolving in a real Hilbert pivot space~$H=H'$. Here~$A\in\clL(V,V')$ and~$A_{\rm rc}=A_{\rm rc}(t)\in\clL(V,H)+\clL(H,V')$ are, respectively,
a time-independent linear diffusion-like operator and a
time-dependent linear reaction--convection-like operator. Above, $V$ is another Hilbert space with continuous and dense inclusion~$V\subseteq H$. 

When~\eqref{sys-y-intro} is unstable we need to use a control in order to achieve the desired stability.  For a given set
\begin{equation}\label{UM}
U_{M}\coloneqq\{\Phi_j\mid 1\le j\le M\}\subset \fkU
\end{equation}
of~$M$ actuators~$\Phi_j$ in a suitable vector space~$\fkU$, we want to find a control input~$u(t)=(u_1(t),\dots,u_{M}(t))\in\bbR^{M}$
such that the state~$y(t)$ of the system
\begin{align}\label{sys-y-intro-K} 
 \dot y +Ay+A_{\rm rc}y =B u,\quad\mbox{where}\quad Bu(t)\coloneqq U_M^\diamond u(t)\coloneqq\sum_{j=1}^{M}u_j(t)\Phi_j,\qquad y(0)= y_0,
\end{align}
converges exponentially to zero as time diverges to~$\infty$.
In other words, we want that for suitable  real numbers $\varrho\ge1$ and~$\mu>0$, 
the solution of~\eqref{sys-y-intro-K} satisfies
\begin{align}\label{goal-intro}
 \norm{y(t)}{H}\le \varrho\ex^{-\mu t}\norm{y_0}{H},\mbox{ for all } y_0\in H,\; t\ge0.
\end{align}

Furthermore, we would like to find the input such that the pair~$(y,u)$, subject to~\eqref{sys-y-intro-K}, minimizes the functional
\begin{align}\label{cost-intro}
\clJ_\infty(y,u;0,y_0)\coloneqq \frac12\int_0^{\infty}\norm{Qy(t)}{H}^2+\norm{u(t)}{\bbR^{M}}^2\,\rmd t,
\end{align}
representing the energy spent during the stabilization proccess, for an appropriate state-penalization operator~$Q\colon H\mapsto H_Q$, where~$H_Q$ is another Hilbert space.

\subsection{Observer design}
Due to robust properties of feedback inputs, we aim at finding the input as a function of the state, $u(t)\coloneqq K(t,y(t))$. When the initial state~$y_0\coloneqq y(0)\in H$, at time~$t=0$, is not known in its entirety, which is the case in many real world applications, in particular, when the state is {\em infinite}-dimensional as for models governed by parabolic partial differential equations, we will need an estimate for~$y_0\coloneqq y(0)\in H$ as well as for the state~$y(t)$ at future instants of time.  To construct such an estimate~$\widehat y(t)$, we use  the measurements provided by a  {\em finite} set 
\begin{equation}\label{WS}
W_S=\{\Psi_j\mid 1\le j\le S\}\subset \fkW
\end{equation}
of~$S$ sensors~$\Psi_j$, in a suitable vector space~$\fkW$.

\subsection{Output based receding-horizon control}
In our setting, under a stabilizability assumption, it is possible to find a feedback relation~$u(t)=K(t,y(t))$ for the minimizer of  the infinite time-horizon (ITH) cost functional~$\clJ_\infty$ in~\eqref{cost-intro}, where~$K=-(U_M^\diamond)^*\Pi(t)$ with~$\Pi$ solving a differential Riccati equation. In the particular case of autonomous and time-periodic dynamics it is also possible to compute an approximation of~$\Pi$ (for a moderate number of degrees of freedom of the numerical spatial approximation of the dymamics). However,  for general nonautonomous systems, it is not possible to compute~$\Pi$ (since we would need to ``solve'' the  differential Riccati equation  backwards in time), and it is not trivial (if possible) to define/quantify what  could be a good approximation of~$\Pi$.

In order to deal with general nonautonomous systems, we look for an ``approximation'' of the optimal infinite-horizon pair~$(y,u)$ minimizing~\eqref{cost-intro}, by following a {\em full-state based} receding-horizon control (RHC) framework, solving a sequence of finite time-horizon (FTH) optimal control problems associated with cost functional versions of~\eqref{cost-intro}, defined in bounded time intervals
\begin{equation}\label{It0T}
\clI_{ {t_{\rm in}} ,T}\coloneqq( {t_{\rm in}} , {t_{\rm in}} +T),\qquad  {t_{\rm in}} \in\{t_n\mid n\in\bbN\},
\end{equation}
for a sequence of concatenation time instants satisfying
\begin{equation}\notag
0=t_0\le t_n<  t_{n+1}\quad\mbox{and}\quad\tau\le t_{n+1}- t_n\le T,
\end{equation}
for an a-priori fixed minimal sampling time~$\tau>0$ and prediction horizon~$T\ge\tau$, namely,
\begin{align}\label{cost-T}
 \clJ_{T}(y,u; {t_{\rm in}} , {y_{\rm in}} )\coloneqq\frac12\int_{ {t_{\rm in}} }^{ {t_{\rm in}} +T}\norm{Qy(t)}{H}^2 +\norm{u(t)}{\bbR^{M}}^2\,\rmd t,
\end{align}
with~$(y,u)$ subject to the dynamics in~\eqref{sys-y-intro-K} for~$t\in \clI_{ {t_{\rm in}} ,T}$,   and  the initial state~$ {y_{\rm in}} =y( {t_{\rm in}} )$.

Within this stabilization approach, the control input is computed/constructed in the time intervals~$\clI_{t_n,T}$ and used in the subintervals~$(t_n,t_{n+1})$, with~$t_{n+1}\ge t_n+\tau$. Note that, to compute the input~$u(t)=u(y( {t_{\rm in}} );t)$ minimizing~\eqref{cost-T}, we need to know the state~$y( {t_{\rm in}} )$ at initial time~$t= {t_{\rm in}} $, which is (assumed to be) not entirely available.

Thus, in practice, we will need an estimate~$\widehat y(t_n)$ for~$y(t_n)$. Once we have such an estimate we can use the input~$u(t)=u(\widehat y(t_n);t)$ instead, for~$t\in (t_n,t_{n+1})$.
To construct such an estimate we will design a dynamical Luenberger observer, based on the output
$w(t)\in\bbR^{S}$ of measurements performed by the sensors~$\Psi_j$. Therefore, we will follow an {\em output based} receding-horizon control (ORHC) framework.

\subsection{The algorithm}
By introducing the Hilbert space
\begin{subequations}\label{spaceXn}
\begin{align}
 &\clX_{ {t_{\rm in}} ,T}\coloneqq  W(\clI_{ {t_{\rm in}} ,T},V,V')\times L^2(\clI_{ {t_{\rm in}} ,T},\bbR^{M})
 \intertext{and, for each~$ {y_{\rm in}} \in H$, its subset}
 & \clX_{ {t_{\rm in}} ,T; {y_{\rm in}} }\coloneqq\{(w,v)\in \clX_{ {t_{\rm in}} ,T}\mid \dot w +Aw+A_{\rm rc}w =U_M^\diamond v\quad\mbox{and}\quad w(0)= {y_{\rm in}} \},
\end{align}
\end{subequations}
we consider the minimization problem as follows.
\begin{problem}\label{P:finT}
Find~$(y^*,u^*)\in\argmin\limits_{(w,v)\in\clX_{ {t_{\rm in}} ,T; {y_{\rm in}} }}  \clJ_{T}(w,v; {t_{\rm in}} , {y_{\rm in}} )$.
\end{problem}
We can show that the minimizer of Problem~\ref{P:finT} is unique, which we denote by
\begin{equation}\label{minT}
(y^*,u^*)=(y_T^*,u_T^*)(t;  {t_{\rm in}} ,  {y_{\rm in}} )=(y_T^*(t;  {t_{\rm in}} ,  {y_{\rm in}} ),u_T^*(t;  {t_{\rm in}} ,  {y_{\rm in}} )),\qquad t\in \clI_{ {t_{\rm in}} ,T}.
\end{equation}

The ORHC associated state-input pair~$(y_{\rm orh},u_{\rm orh})$ will satisfy
\begin{align}\label{systRHCtildeu}
& \dot y_{\rm orh} +Ay_{\rm orh}+A_{\rm rc}y_{\rm orh} =U_M^\diamond u_{\rm orh},\qquad w_{\rm orh}=\clZ  y_{\rm orh},
\end{align}
with unknown initial state~$y_{\rm orh}(0)=y_0$ and 
\begin{align}\notag
&(y_{\rm orh},u_{\rm orh})\rest{(t_n,t_{n+1})}=(y_T^*,u_T^*)(t; t_n, \widehat y(t_n))\rest{(t_n,t_{n+1})},\qquad n\in\bbN,
\end{align}
where~$\widehat y(t_n)$ is an estimate for the state~$y_{\rm orh}(t_n)$ at time~$t=t_n$, provided by a full-dimensional dynamical Luenberger exponential observer based on the output 
\begin{equation}\label{wZy}
w_{\rm orh}(t)=\clZ  y_{\rm orh}(t)\coloneqq(\Psi_1 y_{\rm orh}(t),\dots,\Psi_{S} y_{\rm orh}(t))
\end{equation}
of sensor measurements~$\Psi_j y_{\rm orh}(t)$. The observer takes the form
\begin{align}\label{sys-haty} 
 \dot {\widehat y} +A\widehat y+A_{\rm rc}\widehat y =U_M^\diamond u_{\rm orh}+\fkI_{S}(\clZ   \widehat y-w_{\rm orh}),\qquad
 \widehat y(0)  =\widehat y_0,
\end{align}
where~$\widehat y_0$ is an initial guess/estimate we might have for the unknown initial state~$y_0$, and the output injection operator
\[
\fkI_{S}=\fkI_{S}(t)\colon \bbR^{S}\to H
\]
represents the way the measured output is used/injected into the dynamics of the observer.
The choice of~$\fkI_{S}$ is at our disposal and we want to choose it so that for all~$(\widehat y_0,y_0)\in H\times H$,
\begin{align}\label{goal-obs}
 \norm{\widehat y(t)-y_{\rm orh}(t)}{H}\le C_1\ex^{-\mu_1 t}\norm{\widehat y_0-y_0}{H},\quad t\ge0.
\end{align}

The ORHC steps are illustrated in Algorithm~\ref{alg:rhc+obs}.
\begin{algorithm}[ht]
 \caption{Output based receding horizon control -- ORHC$(T,\tau;\clZ,\widehat y_0)$}
\begin{algorithmic}[1]\label{alg:rhc+obs}
\REQUIRE{Operators~$(A,A_{\rm rc})$ from model; set of actuators~$U_M$;  output operator~$\clZ$;\\ output injection operator~$\fkI_S$; state penalization operator~$Q$ from energy functional;\\ receding-horizon parameters~$T\ge\tau>0$; guess~$\widehat y_0$ for initial state; $T_\infty\in\bbR_+\cup\{\infty\}$;} 
\ENSURE{input~$u_{\rm orh}$; nondecreasing sequence~$(t_n)_{n\in\bbN}$.}
\STATE Set $ {t_{\rm in}} =0$, $ {y_{\rm in}} =\widehat y_0$;
\STATE $t_0\gets 0$; $n\gets 0$;
\WHILE{$ {t_{\rm in}} <T_\infty$}
\STATE\label{alg:st:comp-oc} Compute~$(y_T^*,u_T^*)=(y_T^*,u_T^*)(t;  {t_{\rm in}} ,  {y_{\rm in}} )$, for~$t\in[ {t_{\rm in}} , {t_{\rm in}} +T]$;
\STATE $n\gets n+1$;
\STATE\label{alg:st:tn} $t_n\gets\max\argmin\limits_{t\in[ {t_{\rm in}} +\tau, {t_{\rm in}} +T]} \norm{y_T^*(t;  {t_{\rm in}} ,  {y_{\rm in}} )}{H}$;
\STATE $u_{\rm orh}\rest{[ {t_{\rm in}} ,t_n]}\gets u_T^*(\Bigcdot;  {t_{\rm in}} ,  {y_{\rm in}} )\rest{[ {t_{\rm in}} ,t_n]}$;
\STATE\label{alg:st:realsyst} Let the controlled system~\eqref{systRHCtildeu} evolve for time~$t\in[ {t_{\rm in}} ,t_n]$ with the input~$u_{\rm orh}$, and store the output~$\clZ y_{\rm orh}$ of the sensor measurements;
\STATE\label{alg:st:obssyst} Use the stored output to solve~\eqref{sys-haty} for the estimate~$\widehat y$, for time~$t\in[ {t_{\rm in}} ,t_n]$;
\STATE  $ {t_{\rm in}} \gets t_n$;\quad $ {y_{\rm in}} \gets\widehat y(t_n)$;
 \ENDWHILE
 \end{algorithmic}
\end{algorithm}

\subsection{Observations}
We make the following comments on the steps within Algorithm~\ref{alg:rhc+obs}.
 \begin{enumerate}[noitemsep,topsep=5pt,parsep=5pt,partopsep=15pt,leftmargin=0em]
 \renewcommand{\theenumi}{\underline{\bf O.{\roman{enumi}}}\;} 
 \renewcommand{\labelenumi}{}
\item\theenumi\label{rem:alg:delta} The choice of the sampling time~$\tau_n\coloneqq t_{n}-t_{n-1}\ge\tau$ is made online in Algorithm~\ref{alg:rhc+obs}. Often in the literature, the sampling time is taken simply as~$\tau_n=\tau$, thus with concatenation time instants taken as multiples~$t_n=n\tau$; see~\cite[Intro.~, Alg.~1.1]{AzmiKunisch19}\cite[Sect.~2.4, Alg.~1]{KunPfeiffer20}. The varying online choice made within the algorithm,  is motivated by the need to guarantee a suitable squeezing property at time~$t_n$, $n\ge1$, for the norm of the optimal state~$y_T^*$; see Lemma~\ref{L:optTsqueez}. In this manuscript, we do not investigate the possibility of taking simply~$\tau_n=\tau$;  a proof of the stabilizing property of the input in this setting will require different arguments, because for small~$\tau$ we will not necessarily have the mentioned squeezing property, no matter how large we take~$T$. This can be concluded from the facts that the norm of the optimal state of the ITH problem is not necessarily decreasing and that the optimal ITH state itself is a limit of a sequence of optimal FTH states~\cite{Rod23-scl}.
\item\theenumi\label{rem:alg:st:comp-oc} In step~\ref{alg:st:comp-oc} of Algorithm~\ref{alg:rhc+obs} we use only the available state estimate~$\widehat y(t_{\rm in})$ at initial concatenation time~$ {t_{\rm in}} $. Thus, the computation can be indeed performed in applications where the state~$y(t_{\rm in})$ is not fully available.
\item\theenumi\, In  step~\ref{alg:st:tn} of Algorithm~\ref{alg:rhc+obs}:
\begin{itemize}
\item the choice~$t_n\coloneqq\max\argmin_{t\in[ {t_{\rm in}} +\tau, {t_{\rm in}} +T]} \norm{y_T^*(t;  {t_{\rm in}} ,  {y_{\rm in}} )}{H}$ is made in order to use the computed input for a maximal time-length, this can be seen as an attempt at reducing the number of computed finite-horizon optimal control problems,  thus,  in order to reduce the computational effort;
\item in general, $t_n$ will be strictly smaller that~$ {t_{\rm in}} +T=t_{n-1}+T$, in particular for unstable dynamics and due to the lack of final time penalization in~\eqref{cost-T}, which implies that the optimal input vanishes at time~$ {t_{\rm in}} +T$;
\item we can also choose~$t_n=\min\argmin_{t\in[ {t_{\rm in}} +\tau, {t_{\rm in}} +T]} \norm{y_T^*(t;  {t_{\rm in}} ,  {y_{\rm in}} )}{H}$, for an apriori fixed $\tau\in(0,T]$; this can be seen as an attempt to increase the stabilization exponential rate. 
\end{itemize}
\item\theenumi\, In step~\ref{alg:st:realsyst} of Algorithm~\ref{alg:rhc+obs}, in practice we have just to store the  output~$\clZ y_{\rm orh}$  of measurements provided by the sensors. That is, the system/evolution process is evolving in its plant and we have only access to the output, for example, from measurements of the temperature in a room, provided by sensors/thermostats located in that room; the output is then send to a device/machine which performs the algorithm steps and sends the input to the actuators/radiators located in the room. The computation on this input can be performed remotely, that is, the device/machine can be located in another room/building. 
\end{enumerate}

\subsection{Theoretical strategy}\label{sS:sepprinciple.strat}

Gathering~\eqref{sys-haty} and~\eqref{systRHCtildeu} we have the coupled system
\begin{subequations}\label{sys-coupled} 
\begin{align}
 & \dot y_{\rm orh} +Ay_{\rm orh}+A_{\rm rc}y_{\rm orh} =U_M^\diamond u_{\rm orh},\qquad&& w_{\rm orh}=\clZ  y_{\rm orh},\\
  &\dot {\widehat y} +A\widehat y+A_{\rm rc}\widehat y =U_M^\diamond u_{\rm orh}+\fkI_{S}(\clZ   \widehat y-w_{\rm orh}),\qquad&&
 \widehat y(0)  \!=\!\widehat y_0,\label{sys-coupled-haty} 
\end{align}
\end{subequations}
where the initial state~$y_{\rm orh}(0)$ is unknown. To prove the aimed stabilizing property of the ORHC input~$u_{\rm orh}$ obtained by Algorithm~\ref{alg:rhc+obs}, we shall follow a standard argument using, in particular, the fact that the dynamics of the estimate error~$z\coloneqq \widehat y-y_{\rm orh}$ decouples from the dynamics of the controlled state (separation principle). Indeed, for~$(z,y_{\rm orh})\coloneqq(\widehat y-y_{\rm orh},y_{\rm orh})$, from~\eqref{sys-coupled}, we obtain
\begin{subequations}\label{sys-coupled-error}
\begin{align}
 \dot y_{\rm orh} +Ay_{\rm orh}+A_{\rm rc}y_{\rm orh} &=U_M^\diamond u_{\rm orh},&&\quad  y_{\rm orh}(0)= y_0;\label{sys-coupled-error-y}\\
  \dot z +Az+A_{\rm rc}z &=\fkI_{S}\clZ    z,\label{sys-coupled-error-z}
 &&\quad z(0)=z_0;
\end{align}
\end{subequations}
where~$y_0$ and~$z_0\coloneqq\widehat y_0-y_0$ are unknown and we can see that the dynamics of the state estimate error~$z$ is independent of the controlled state~$y_{\rm orh}$.

The stabilizability result  shall be derived firstly, in Section~\ref{S:stabilizability}, under some key general/abstract assumptions. Then, the satisfiability of the assumptions is shown in Section~\ref{S:illust-parabolic}, in the context of concrete parabolic equations.

\section{Stabilizing property of the ORHC input}\label{S:stabilizability}
Let~$H=H'$ be a real separable pivot Hilbert space and let~$V\subset H$ be another real separable Hilbert space.

Hereafter, we write~$\bbR$ and~$\bbN$ for the sets of real numbers and nonnegative
integers, respectively, and  their subsets of positive numbers~$\bbR_+\coloneqq(0,+\infty)$
and $\bbN_+\coloneqq\mathbb N\setminus\{0\}$.

\subsection{Stabilizability and detectability}
Let us assume that the solutions of the following nonautonomous system, with~$\fkA(t)\in\clL(V,V')$
\begin{equation}\label{sys.fkL}
\dot y=\fkA y,\qquad y(0)=y_0\in H,
\end{equation}
satisfy~$y\in L^2((0,T),V)$ and~$\dot y\in L^2((0,T),V')$, for all~$T>0$, and~$y\in C(\overline\bbR_+,H)$.

\begin{definition}
The operator~$\fkA\in L^\infty(\bbR_+,\clL(V,V'))$ is said $(\varrho,\mu)$-exponentially stable, with constants~$\varrho\ge1$ and~$\mu>0$, if every weak solution of~\eqref{sys.fkL}
satisfies~\eqref{goal-intro}.
\end{definition}

Let $\clH_{1}$ and~$\clH_{2}$ be two additional Hilbert spaces.

\begin{definition}\label{D:stab-pair}
The pair~$(\fkA,\fkB)\in L^\infty(\bbR_+,\clL(V,V'))\times L^\infty(\bbR_+,\clL(\clH_1,H))$ is said $(\varrho,\mu)$-stabilizable, if there is~$\fkK\in L^\infty(\bbR_+, \clL(V,\clH_1))$ so that~$\fkA+\fkB \fkK$
is $(\varrho,\mu)$-exponentially stable.  The pair~$(\fkA,\fkB)$ is said stabilizable,  if it is $(\varrho,\mu)$-stabilizable for some~$\varrho\ge1$ and~$\mu>0$.
\end{definition}

\begin{definition}\label{D:detect-pair}
The pair~$(\fkA, \fkC)\in L^\infty(\bbR_+,\clL(V,V'))\times L^\infty(\bbR_+,\clL(V,\clH_2))$ is said $(\varrho,\mu)$-detectable, if
there is~$\fkL\in  L^\infty(\bbR_+,\clL(\clH_2,H))$ so that~$\fkA+\fkL\fkC$
is $(\varrho,\mu)$-exponentially stable. The pair~$(\fkA, \fkC)$ is said detectable, if it is $(\varrho,\mu)$-detectable for some~$\varrho\ge1$ and~$\mu>0$.
\end{definition}

\subsection{Assumptions}
We gather general assumptions on each of the involved operators, namely, on~$A$ and~$A_{\rm rc}$ defining the free dynamics, and on the additional operators~$U_M^\diamond$, $\clZ$, and~$Q$, used in the controlled dynamics.

For the operators involved in the free dynamics we require the following.
\begin{assumption}\label{A:A}
The inclusion $V\subseteq H$ is dense, continuous, and compact. the operator
$A\in\clL(V,V')$ is symmetric and $(y,z)\mapsto\langle Ay,z\rangle_{V',V}$ is a complete scalar product in~$V.$ 
\end{assumption}

\begin{assumption}\label{A:A1}
For almost every~$t>0$, $A_{\rm rc}(t)\in\clL(H,V')+\clL(V, H)$ 
and we have a uniform bound as $\norm{A_{\rm rc}}{L^\infty(\bbR_+,\clL(H,V')+\clL(V, H))}\eqqcolon C_{\rm rc}<+\infty.$
\end{assumption}

For the set of sensors we require the following, in terms of the output operator~$\clZ$.
\begin{assumption}\label{A:sens}
The pair~$(-Az-A_{\rm rc},\clZ)$ is detectable.
\end{assumption}

For the set of actuators and cost functional we require the following, in terms of the control operator~$U_M^\diamond$ and state-penalization operator~$Q$.
\begin{assumption}\label{A:act}
The pair~$(-Az-A_{\rm rc},U_M^\diamond)$ is stabilizable and the pair~$(-Az-A_{\rm rc},Q)$ is detectable.
\end{assumption}

\subsection{Auxiliary results}
The first auxiliary result is as follows.
\begin{lemma}\label{L:detect}
Let Assumptions~\ref{A:A}, \ref{A:A1}, and~\ref{A:sens} hold true. Then, there exists an output injection operator~$\fkJ_S\in L^\infty(\bbR_+,\clL(\bbR^{S}, V'))$ such that
the solution of the system
\begin{equation}
\dot z+Az+A_{\rm rc}z=\fkJ_S\clZ z,\qquad z(0)=z_0,\notag
\end{equation}
satisfies, for some constants~$C_{\rmo,1}\ge1$ and~$\mu_1>0$ independent of~$z_0$,
 \begin{align}\label{z-expstable}  
 &\norm{z(t)}{H}\le C_{\rmo,1}\rme^{-\mu_1(t-s)}\norm{z(s)}{H},\quad\mbox{for all}\quad t\ge s\ge0.
 \end{align} 
\end{lemma}
\begin{proof}
Straightforward from regularity of parabolic-like equations and Definition~\ref{D:detect-pair}.
\end{proof}

The next auxiliary result is as follows.
\begin{lemma}\label{L:norm-u-optT}
The optimal control~$u_T^*(t)=u_T^*(t;  {t_{\rm in}} ,  {y_{\rm in}} )$ satisfies
\begin{equation}\notag
	\norm{u_T^*}{L^2(( {t_{\rm in}} , {t_{\rm in}} +T),\bbR^M)}\le C_{u}\norm{ {y_{\rm in}} }{H}.
\end{equation}
with~$C_u$ independent of~$ {t_{\rm in}} $ and~$T$.
\end{lemma}
\begin{proof}
We have that
\begin{align}
\norm{u_T^*}{L^2(( {t_{\rm in}} , {t_{\rm in}} +T),\bbR^M)}^2\le 2\clJ_{T}(u_T^*,u_T^*;  {t_{\rm in}} ,  {y_{\rm in}} )
\end{align}
where~$y^*_{T}\coloneqq y_{T}^*(\Bigcdot; {t_{\rm in}} ,  {y_{\rm in}} )$ and~$u_{T}^*\coloneqq u_{T}^*(\Bigcdot;  {t_{\rm in}} ,  {y_{\rm in}} )$.
We also have, by optimality, that
\begin{align*}
\clJ_{T}(y_{T}^*, u_{T}^*;  {t_{\rm in}} ,  {y_{\rm in}} )&\le \clJ_{T}(y_{\infty}^*(\Bigcdot;  {t_{\rm in}} ,  {y_{\rm in}} ),u_{\infty}^*(\Bigcdot;  {t_{\rm in}} ,  {y_{\rm in}} ) ;  {t_{\rm in}} ,  {y_{\rm in}} ) \\    &\le \clJ_{\infty}(y_{\infty}^*(\Bigcdot;   {t_{\rm in}} ,  {y_{\rm in}} ),u_{\infty}^*(\Bigcdot;  {t_{\rm in}} ,  {y_{\rm in}} );  {t_{\rm in}} ,  {y_{\rm in}} ),
\end{align*}
where~$y^*_{\infty}\coloneqq y_{\infty}^*(\Bigcdot; {t_{\rm in}} ,  {y_{\rm in}} )$ and~$u_{\infty}^*\coloneqq u_{\infty}^*(\Bigcdot;  {t_{\rm in}} ,  {y_{\rm in}} )$

Recalling that, under the stabilizability assumption, we have that the optimal cost of the~ITH problem is bounded, namely, that
\begin{align}
\clJ_{\infty}(y_{\infty}^*(\Bigcdot;   {t_{\rm in}} ,  {y_{\rm in}} ),u_{\infty}^*(\Bigcdot;  {t_{\rm in}} ,  {y_{\rm in}} );  {t_{\rm in}} ,  {y_{\rm in}} )\le C_\infty \norm{ {y_{\rm in}} }{H}^2,\label{Cinfty}
\end{align}
we can conclude that the statement follows with~$C_u=C_\infty^\frac12$. The fact that~$C_\infty$ is independent of~$ {t_{\rm in}} $ is a consequence of the uniform bound in Assumption~\ref{A:A1}.
\end{proof}

\begin{lemma}\label{L:optTsqueez}
Let Assumptions~\ref{A:A}, \ref{A:A1}, and~\ref{A:act} hold true. Then, there are constants~$\theta\in(0,1)$ and~$T\ge\tau>0$, with~$T$ large enough, such that for every~$( {t_{\rm in}} ,\widehat y_0)\in \overline\bbR_+\times H$, the optimal pair~$(y_T^*,u_T^*)=(y_T^*,u_T^*)(t;  {t_{\rm in}} ,  {y_{\rm in}} )$ satisfies, 
 \begin{align}\label{yrc-expstable}  
 &\min_{t\in[ {t_{\rm in}} , {t_{\rm in}} +T]}\norm{y_T^*(t;  {t_{\rm in}} ,  {y_{\rm in}} )}{H}\le\theta\norm{ {y_{\rm in}} }{H}\quad\mbox{and}\quad\max\argmin_{t\in[ {t_{\rm in}} , {t_{\rm in}} +T]}\norm{y_T^*(t;  {t_{\rm in}} ,  {y_{\rm in}} )}{H}>\tau.
 \end{align} 
\end{lemma}

\begin{proof}
By the detectability of~$(-A-A_{\rm rc},Q)$, there exists~$G\in L^\infty(\bbR_+,\clL(H_Q,V'))$ such that
\begin{align}\label{sys-z-detect}
\dot z+Az+A_{\rm rc}z=GQz
\end{align}
is exponentially stable. Then, for the optimal pair~$(y^*_T, u^*_T)$ we find
\begin{align}
\dot y^*_T+Ay^*_T+A_{\rm rc}y^*_T=U_M^\diamond u=GQy^*_T -GQy^*_T+U_M^\diamond u^*_T,\notag
\end{align}
Denoting the evolution operator of~\eqref{sys-z-detect} by~$Z(s,t)$, $t\ge s\ge0$, and using Duhamel formula,
\begin{align}
y^*_T(t)=Z( {t_{\rm in}} ,t) {y_{\rm in}} +\int_0^t Z(s,t)(U_M^\diamond u^*_T(s)-GQy^*_T(s))\,\rmd s.\notag
\end{align}
We can see that, for some constants~$C_Z\ge1$, $\mu_Z>0$, and~$D_Z\ge0$, depending on~$(Q,G)$,
\begin{align}
\norm{y^*_T( {t_{\rm in}} +t)}{H}&\le C_Z\rme^{-\mu_Z t}\norm{ {y_{\rm in}} }{H}+C_Z\int_{ {t_{\rm in}} }^t \rme^{-\mu_Z (t-s)}\norm{U_M^\diamond u^*_T(s)-GQy^*_T(s)}{H}\,\rmd s\notag\\
&\le C_Z\rme^{-\mu_Z t}\norm{ {y_{\rm in}} }{H}+D_Z\int_{ {t_{\rm in}} }^t \rme^{-\mu_Z (t-s)}(\norm{u^*_T(s)}{\bbR^M}+\norm{Qy^*_T(s)}{H_Q})\,\rmd s,\notag
\end{align}
which implies
\begin{align}
\norm{y^*_T( {t_{\rm in}} +t)}{H}^2
&\le 2C_Z^2\rme^{-\mu_Z t}\norm{ {y_{\rm in}} }{H}^2+2D_Z^2\left(\int_{ {t_{\rm in}} }^t \rme^{-\mu_Z (t-s)}(\norm{u^*_T(s)}{\bbR^M}+\norm{Qy^*_T(s)}{H_Q})\,\rmd s\right)^2.\notag
\end{align}
By time integration over~$\clI_{ {t_{\rm in}} }\coloneqq( {t_{\rm in}} , {t_{\rm in}} +T)$, Young convolution inequality~\cite[Eq.~(2)]{Beckner75} \cite[Eq.~(2.22)]{BrascampLieb76}, and optimality, we arrive at
\begin{align}
&\norm{y^*_T}{L^2(\clI_{ {t_{\rm in}} },H)}^2\le C_4\norm{ {y_{\rm in}} }{H}^2+C_4\norm{\rme^{-\mu_Z(\Bigcdot- {t_{\rm in}} )}}{L^1(\clI_{ {t_{\rm in}} },\bbR)}^2\norm{\bigl.\norm{u^*_T}{\bbR^M}+\norm{Qy^*_T}{H_Q}\bigr.}{L^2(\clI_{ {t_{\rm in}} },\bbR)}^2\notag\\
&\hspace{1em}\le C_4\norm{ {y_{\rm in}} }{H}^2+C_5\norm{\norm{u^*_T}{\bbR^M}^2+\norm{Qy^*_T}{H_Q}^2}{L^1(\clI_{ {t_{\rm in} }},\bbR)}\notag\\
&\hspace{1em}\le C_4\norm{ {y_{\rm in}} }{H}^2+C_5   \norm{\norm{u_\infty^*}{\bbR^M}^2+\norm{Qy_\infty^*}{H_Q}^2}{L^1(( {t_{\rm in}} ,\infty),\bbR)}\le (C_4+C_5C_\infty)\norm{ {y_{\rm in}} }{H_Q}^2,\label{bound-bqueez}
\end{align}
where we used~\eqref{Cinfty}. Therefore,  with~$C_6\coloneqq (C_4+C_5C_\infty)$, for a given~$\tau\in(0,T)$, we must have
$\norm{y^*_T(T^\circ)}{H}^2\le\frac1{T-\tau}C_6\norm{ {y_{\rm in}} }{H}^2$ for some~$T^\circ\in[ {t_{\rm in}} +\tau, {t_{\rm in}} +T]$,  since otherwise we would find
\begin{align}
\norm{y^*_T}{L^2(\clI_{ {t_{\rm in}} },H)}^2&\ge\norm{y^*_T}{L^2(( {t_{\rm in}} +\tau, {t_{\rm in}} +T),H)}^2>\tfrac1{T-\tau}C_6\norm{ {y_{\rm in}} }{H}^2(T-\tau)=C_6\norm{ {y_{\rm in}} }{H}^2,\notag
\end{align}
 which would contradict~\eqref{bound-bqueez}. Now, by choosing~$T>\tau+\theta^{-1}C_6$ we find that that~$\frac1{T-\tau}C_6<\theta$.
 Therefore, the result follows for~$T=\ovlineC{C_Z,D_Z,\tau,C_\infty,\theta^{-1}}$ large enough.
\end{proof}

\subsection{Main result}
We are now able to state the main result of this manuscript is as follows.
\begin{theorem}\label{T:main-y}
Let Assumptions~\ref{A:A}--\ref{A:act} hold true. Let~$\mu_1$ and~$\theta$ be as in~\eqref{z-expstable} and~\eqref{yrc-expstable}, and let~$\mu_2<\min\{\mu_1,T^{-1}\log(\theta^{-1})\}$. Then, there are constants~$D_1\ge 1$ and~$D_2>0$, such that the state~$y_{\rm orh}$ associated with the input~$u_{\rm orh}={\rm ORHC}(\tau,T;\clZ,\widehat y_0)$, provided by Algorithm~\ref{alg:rhc+obs}, satisfies
 \begin{align}\label{y-expstable}  
 &\norm{ (y_{\rm orh}(t),z(t))}{H\times H}\le D_1\rme^{-\mu_2t}\norm{ (y_{\rm orh}(0),z(0))}{H\times H},\quad\mbox{for all}\quad t\ge0,
 \end{align} 
 for every guess~$\widehat y_0$, with the control input satisfying~$\norm{u_{\rm orh}}{L^2(\bbR_+,\bbR^{M})}\le D_2\norm{y(0),z(0)}{H\times H}$.\\
In the case~$\mu_1\ne T^{-1}\log(\theta^{-1})$ we can take~$\mu_2=\min\{\mu_1,T^{-1}\log(\theta^{-1})\}$.
\end{theorem}

\begin{proof}
 For~$t\in\clI_{n}\coloneqq(t_n,t_{n+1})$, we compare the solution of~\eqref{sys-coupled-error-y}, satisfying
\begin{align}
 \dot y_{\rm orh} +Ay_{\rm orh}+A_{\rm rc}y_{\rm orh} &=U_M^\diamond u_{\rm orh}
,\notag
\end{align}
where~$y_{\rm orh} 
(t_{n})$ is unknown and~$u_{\rm orh}=u_{\rm orh}(\tau,T;\clZ,\widehat y(t_{n}))$ is the input given by Algorithm~\ref{alg:rhc+obs},
with the solution of
\begin{align}\label{sys.bary}
 \dot {\underline y} +A{\underline y} +A_{\rm rc}{\underline y}  &=U_M^\diamond u_{\rm orh},\qquad {\underline y} 
(t_{n})=\widehat y(t_{n}),
\end{align}
For the difference~$\chi\coloneqq y_{\rm orh}-\underline y$ we find
\begin{align}
 \dot \chi +A\chi +A_{\rm rc}\chi  &=0,\qquad \chi(t_{n})=\chi_n\coloneqq y_{\rm orh}(t_{n})-\widehat y(t_{n}),\notag
\end{align}
and~$\norm{\chi(t)}{H}\le C_1\norm{\chi_n}{H}$, for some constant~$C_1=\ovlineC{\tau}\ge1$ and all~$t\in\clI_{n}$. Hence,
\begin{align}\notag
\norm{ y_{\rm orh}(t)}{H}=\norm{ \chi(t)+\underline y(t)}{H}\le C_1\norm{y_{\rm orh}(t_{n})-\widehat y(t_{n})}{H}+\norm{\underline y(t)}{H}.
\end{align}

Next, note that, within system~\eqref{sys.bary}, the input coincides with  the optimal input~$u_{\rm orh}(t)=u_T^*(t;  {t_{\rm in}} ,  {y_{\rm in}} )$, for~$t\in\clI_{n}$.
Therefore, by Lemma~\ref{L:optTsqueez}, it follows that
\begin{align}
\norm{ y_{\rm orh}(t_{n+1})}{H}&\le C_1\norm{y_{\rm orh}(t_n)-\widehat y(t_n)}{H}+\theta\norm{\widehat y(t_n)}{H}\notag\\
&\le (C_1+\theta)\norm{y_{\rm orh}(t_{n})-\widehat y(t_{n})}{H}+\theta\norm{y_{\rm orh}(t_n)}{H}.\notag%
\end{align}

Recalling that~$z= \widehat y-y_{\rm orh}$ satisfies the dynamics in~\eqref{sys-coupled-error-z}, by 
Lemma~\ref{L:detect} we obtain, in particular, with~$C_2=(C_1+\theta)C_{\rmo,1}$,
\begin{align}
\norm{ y_{\rm orh}(t_{n+1})}{H}&\le C_2\rme^{-\mu_1t_n}\norm{z(0)}{H}+\theta\norm{y_{\rm orh}(t_{n})}{H}.\label{sq.yorh1}
\end{align}

Next, we iterate~\eqref{sq.yorh1}, to obtain
\begin{align}
\norm{ y_{\rm orh}(t_{n+1})}{H}&\le C_2\rme^{-\mu_1t_{n}}\norm{z(0)}{H}+\theta\left(C_2\rme^{-\mu_1t_{n-1}}\norm{z(0)}{H}+\theta\norm{y_{\rm orh}(t_{n-1})}{H}\right)\notag\\
&= C_2(\rme^{-\mu_1t_{n}}+\theta\rme^{-\mu_1t_{n-1}})\norm{z(0)}{H}+\theta^{2}\norm{y_{\rm orh}(t_{n-1})}{H}\notag\\
&\le C_2\Bigl({\textstyle\sum\limits_{j=0}^n}\theta^j\rme^{-\mu_1t_{n-j}}\Bigr)\norm{z(0)}{H}+\theta^{n+1}\norm{y_{\rm orh}(0)}{H},\notag
\end{align}
and observe that~$t_{n-j}\ge(n-j)\tau$, which leads to
\begin{align}
\norm{ y_{\rm orh}(t_{n+1})}{H}&\le C_2\Bigl({\textstyle\sum\limits_{j=0}^n}\theta^j\rme^{-\mu_1(n-j)\tau}\Bigr)\norm{z(0)}{H}+\theta^{n+1}\norm{y_{\rm orh}(0)}{H}.\label{sq.yorh2}
\end{align}
Next, we consider the three cases~$\theta\rme^{\mu_1\tau}<1$, $\theta\rme^{\mu_1\tau}>1$, and~$\theta\rme^{\mu_1\tau}=1$ separately.

\begin{subequations}\label{sq.yorh2.x}
\noindent$\bullet$ Case~$\theta\rme^{\mu_1\tau}<1$. We write~$\theta^j\rme^{-\mu_1(n-j)\tau}=(\theta\rme^{\mu_1\tau})^j\rme^{-\mu_1n\tau}$ and, from~\eqref{sq.yorh2}, 
\begin{align}
\norm{ y_{\rm orh}(t_{n+1})}{H}&\le C_2(1-\theta\rme^{\mu_1\tau})^{-1}\rme^{\mu_1\tau}\rme^{-\mu_1\tau(n+1)}\norm{z(0)}{H}+\theta^{n+1}\norm{y_{\rm orh}(0)}{H}.\label{sq.yorh2.1}
\end{align}

\noindent$\bullet$ Case~$\theta\rme^{\mu_1\tau}>1$. We write~$\theta^j\rme^{-\mu_1(n-j)\tau}=\theta^n(\theta^{-1}\rme^{-\mu_1\tau})^{n-j}$ and, from~\eqref{sq.yorh2},
\begin{align}
\norm{ y_{\rm orh}(t_{n+1})}{H}&\le  C_2(1-\theta^{-1}\rme^{-\mu_1\tau})^{-1}\theta^{-1}\theta^{n+1}\norm{z(0)}{H}+\theta^{n+1}\norm{y_{\rm orh}(0)}{H}\label{sq.yorh2.2}
\end{align}

\noindent$\bullet$ Case~$\theta\rme^{\mu_1\tau}=1$. We fix an arbitrary~$\varepsilon\in(0,\mu_1)$ and write~$\theta^j\rme^{-\mu_1(n-j)\tau}=\rme^{-\mu_1n\tau}=\rme^{-\varepsilon\tau n}\rme^{-(\mu_1-\varepsilon)\tau n}$ and, from~\eqref{sq.yorh2}, 
\begin{align}
\norm{ y_{\rm orh}(t_{n+1})}{H}&\le  C_2(n+1)\rme^{-\varepsilon\tau n}\rme^{-(\mu_1-\varepsilon)\tau n}\norm{z(0)}{H}+\theta^{n+1}\norm{y_{\rm orh}(0)}{H}\notag\\
&\hspace{-3em}\le C_2(1+\max\{n\rme^{-\varepsilon\tau n}\mid n\in\bbN\})\rme^{-(\mu_1-\varepsilon)n\tau}\norm{z(0)}{H}+\theta^{n+1}\norm{y_{\rm orh}(0)}{H}\notag\\
&\hspace{-3em}\le C_2(1+(\varepsilon\tau)^{-1}\rme^{-1})\rme^{-(\mu_1-\varepsilon)n\tau}\norm{z(0)}{H}+\theta^{n+1}\norm{y_{\rm orh}(0)}{H}\notag\\
&\hspace{-3em}\le C_2(1+(\varepsilon\tau)^{-1}\rme^{-1})\rme^{(\mu_1-\varepsilon)\tau}\rme^{-(\mu_1-\varepsilon)\tau(n+1)}\norm{z(0)}{H}+\theta^{n+1}\norm{y_{\rm orh}(0)}{H}.\label{sq.yorh2.3}
\end{align}
\end{subequations}
Therefore, in either case~$\theta\rme^{\mu_1\tau}>0$, by~\eqref{sq.yorh2.x} we conclude that
\begin{subequations}\label{sq.yorh3}
\begin{align}
&\norm{ y_{\rm orh}(t_{n+1})}{H}\le  C_2D\theta_1^{n+1}\norm{z(0)}{H}+\theta^{n+1}\norm{y_{\rm orh}(0)}{H},\\
&\qquad\mbox{with}\quad\theta_1\coloneqq\begin{cases}
\max\{\theta,\rme^{-\mu_1\tau}\}&\mbox{if }\theta\ne\rme^{-\mu_1\tau},\\
\rme^{-(\mu_1-\varepsilon)\tau}&\mbox{if }\theta=\rme^{-\mu_1\tau},
\end{cases}
\end{align}
\end{subequations}
and a constant $D>0$.

Now, for~$t\in\overline \clI_{n+1}=[t_{n+1},t_{n+2}]$, we find, with some constant~$C_3=\ovlineC{T}$,
\begin{align}
\norm{ y_{\rm orh}(t)}{H}^2&\le C_3\left(\norm{ y_{\rm orh}(t_{n+1})}{H}^2+\norm{ u_{\rm orh}}{L^2(\clI_{n+1},\bbR^M)}^2\right)\notag\\
&\le C_3\left(\norm{ y_{\rm orh}(t_{n+1})}{H}^2+C_u^2\norm{ \widehat y(t_{n+1})}{H}^2\right)\notag\\
&\le C_3\left((1+2C_u^2)\norm{ y_{\rm orh}(t_{n+1})}{H}^2+2C_u^2\norm{y_{\rm orh}(t_{n+1})-\widehat y(t_{n+1})}{H}^2\right)\notag
\end{align}
with~$C_u$ as in Lemma~\ref{L:norm-u-optT}. Recalling again that~$z= \widehat y-y_{\rm orh}$ satisfies the dynamics in~\eqref{sys-coupled-error-z}, by~\eqref{sq.yorh3} and Lemma~\ref{L:detect} we find
\begin{align}
\norm{ y_{\rm orh}(t)}{H}^2&\le C_4\theta_1^{2(n+1)}\norm{z(0)}{H}^{2}+C_5\theta^{2(n+1)}\norm{y_{\rm orh}(0)}{H}^2+C_u^2C_{\circ,1}^2\rme^{-2\mu_1t_{n+1}}\norm{z(0)}{H}^2\notag\\
&\le  C_4\theta_1^{2(n+1)}\norm{z(0)}{H}^{2}+C_5\theta^{2(n+1)}\norm{y_{\rm orh}(0)}{H}^2+C_u^2C_{\circ,1}^2\rme^{-2\mu_1\tau(n+1)}\norm{z(0)}{H}^2
\notag
\end{align}
with~$C_4=2C_3(1+2C_u^2)C_2^2D^{2}$ and~$C_5=2C_3(1+2C_u^2)$.

Next, we note that we have~$\theta_1\ge\max\{\theta,\rme^{-\mu_1\tau}\}$ and
\begin{align}
\theta_1^{2(n+1)}=\rme^{-2(n+1)\log(\theta_1^{-1})}=\rme^{-2(n+1)t_{n+1}^{-1}\log(\theta_1^{-1})t_{n+1}}\le\rme^{-2T^{-1}\log(\theta_1^{-1})t_{n+1}},\notag
\end{align}
due to~$t_{n+1}\le (n+1)T$. Hence, we arrive at
\begin{align}
\norm{ y_{\rm orh}(t)}{H}^2&\le(C_4+C_u^2C_{\circ,1}^2)\rme^{-2\mu_2t_{n+1}}\norm{z(0)}{H}^2+C_5\rme^{-2\mu_2t_{n+1}}\norm{y_{\rm orh}(0)}{H}^2\notag\\
&\le C_6\rme^{-2\mu_2t}\left(\norm{z(0)}{H}^2+\norm{y_{\rm orh}(0)}{H}^2\right)\,\label{sq.yorh-t>delta}\\
&\hspace{-3em}\mbox{with}\quad \mu_2\coloneqq T^{-1}\log(\theta_1^{-1}),\quad\mbox{for}\quad t\in[t_{n+1},t_{n+2}],\quad n\in\bbN,\label{sq.yorh-t>deltat}
\end{align}
where~$C_6=\rme^{2\mu_2T}\max\{C_4+C_u^2C_{\circ,1}^2,C_5\}$.

Next, for~$t\in[0,t_1]$, we obtain
\begin{align}
\norm{ y_{\rm orh}(t)}{H}^2&\le C_3\left(\norm{ y_{\rm orh}(0)}{H}^2+\norm{ u_{\rm orh}}{L^2((0,t_1),H)}^2\right)\le C_3\left(\norm{y_{\rm orh}(0)}{H}^2+C_u^2\norm{ \widehat y(0)}{H}^2\right)\notag\\
&\le  C_3\left((1+2C_u^2)\norm{y_{\rm orh}(0)}{H}^2+2C_u^2\norm{ z(0)}{H}^2\right)\notag\\
&\le  C_3\rme^{2\mu_2T}\rme^{-2\mu_2t}\left((1+2C_u^2)\norm{y_{\rm orh}(0)}{H}^2+2C_u^2\norm{ z(0)}{H}^2\right),\quad\mbox{for}\quad t\in[0,t_1].\label{sq.yorh-t<delta}
\end{align}
By~\eqref{sq.yorh-t>delta} and~\eqref{sq.yorh-t<delta}, it follows that, for some constant~$C_7\ge1$,
\begin{align}
&\norm{ y_{\rm orh}(t)}{H}^2\le C_7\rme^{-2\mu_2t}\left(\norm{z(0)}{H}^2+\norm{y_{\rm orh}(0)}{H}^2\right),\qquad 0\le t<\lim_{n\to+\infty}t_n=\infty.\label{sq.yorh-t>0}
\end{align}
Note that by construction within Algorithm~\ref{alg:rhc+obs}, we have that~$t_n>n\tau\to\infty$.
Using again Lemma~\ref{L:detect}, we arrive at
\begin{align}
\norm{ (y_{\rm orh}(t),z(t))}{H\times H}^2&=\norm{y_{\rm orh}(t)}{H}^2+\norm{ z(t)}{H}^2\notag\\
&\le C_7\rme^{-2\mu_2t}\left(\norm{z(0)}{H}^2+\norm{y_{\rm orh}(0)}{H}^2\right)+C_{\rmo,1}^2\rme^{-2\mu_1 t}\norm{ z(0)}{H}^2\notag\\
&\le C_8 \rme^{-2\mu_2t}\norm{(y_{\rm orh}(0),z(0))}{H\times H}^2,\quad\mbox{for all}\quad t\ge0,\notag
\end{align}
with~$C_8\coloneqq \max\{C_7,C_{\rmo,1}^2\}$, which gives us~\eqref{y-expstable}, with~$D_1\coloneqq C_8^\frac12$.

Finally, for the control input, using Lemma \ref{L:norm-u-optT},   we find the estimates
\begin{align}
\norm{u_{\rm orh}}{L^2((0,\infty),\bbR^{M})}^2&= \sum_{n=0}^{\infty}\norm{u_{\rm orh}}{L^2(\clI_{n},\bbR^{M})}^2\le C^2_{u}\sum_{n=0}^{\infty}\norm{\widehat y(t_n)}{H}^2\notag\\
&\le 2C^2_{u}\sum_{n=0}^{\infty}\left(\norm{y_{\rm orh}(t_n)}{H}^2+\norm{z(t_n)}{H}^2\right)\le 2C^2_{u}C_8\norm{y(0),z(0)}{H\times H}^2\sum_{n=0}^{\infty}\rme^{-2t_n}\notag\\
&\le 2C^2_{u}C_8\norm{y(0),z(0)}{H\times H}^2\sum_{n=0}^{\infty}\rme^{-2n\tau}= D_2^2\norm{y(0),z(0)}{H\times H}^2,\notag
\end{align}
with~$D_2\coloneqq (2C^2_{u}C_8(1-\rme^{-2\tau})^{-1})^\frac12$,
which finishes the proof.
\end{proof}

 \section{Example of application}\label{S:illust-parabolic}
The results in the previous section  can be applied to concrete parabolic equations,  as
\begin{subequations}\label{sys-yhaty-parab-CLoop} 
\begin{align}
&\tfrac{\p}{\p t} \widehat y -\nu\Delta \widehat y+a\widehat y +b\cdot\nabla \widehat y
=U_{M}^\diamond u+\fkI_S(\clZ\widehat y-\clZ    y),\\
 &\tfrac{\p}{\p t} y -\nu\Delta y+ay +b\cdot\nabla y
 =U_{M}^\diamond u,\\
  &\fkT y\rest{\Gamma}=0=\fkT\widehat y\rest{\Gamma},\quad
(y(0),\widehat y(0))=(y_0,\widehat y_0),
       \end{align}
 \end{subequations}
with the state~$y$ defined in a bounded open convex polygonal/polyhedral  spatial domain~$\Omega\subset \bbR^d$, where~$d$ is a positive integer (in applications, often~$d\in\{1,2,3\}$).
The controlled state~$y=y(x,t)$ and its estimate~$\widehat y=\widehat y(x,t)$ are
 functions defined for~$(x,t)\in\Omega\times(0,+\infty)$. The operator~$\fkT$
imposes the boundary conditions
at the boundary~$\Gamma=\p\Omega$ of~$\Omega$,
\begin{align}
  \fkT &=\Id,&&\quad\mbox{for Dirichlet boundary conditions},\notag\\
  \fkT &=\bfn\cdot\nabla,&&\quad\mbox{for Neumann boundary conditions,}\notag
\end{align}
where~$\bfn=\bfn(\bar x)$ stands for the outward unit normal vector to~$\Gamma$, at~$\bar x\in\Gamma$.

The functions~$a=a(x,t)$ and $b=b(x,t)$, 
defined in~$\Omega\times(0,+\infty)$ satisfy
\begin{align}
 &a\in L^\infty(\Omega\!\times\!(0,+\infty)),\;\; b\in L^\infty(\Omega\!\times\!(0,+\infty))^d.
 \label{assum.abf.parab}
\end{align}

By defining, for both Dirichlet and Neumann boundary conditions, the spaces
\begin{align}
H^1_\fkT(\Omega)&\coloneqq\begin{cases}
 \{h\in  H^1(\Omega)\mid h\rest\Gamma=0\},\quad&\mbox{if }\fkT=\Id;\\
H^1(\Omega),\quad&\mbox{if }\fkT=\bfn\cdot\nabla;
\end{cases}\\
H^2_\fkT(\Omega)&\coloneqq\{h\in  H^2(\Omega)\mid \fkT h\rest\Gamma=0\},
\end{align}
we can write~\eqref{sys-yhaty-parab-CLoop} in the abstract form~\eqref{sys-coupled}. For this purpose, we set
\begin{equation}\notag
 H\coloneqq L^2(\Omega),\quad V\coloneqq H^1_{\fkT}(\Omega),\quad\mbox{and}\quad \rmD(A)\coloneqq H^2_\fkT(\Omega),
\end{equation}
and the linear operators
\begin{equation}\label{A-Lap}
 A\coloneqq -\nu\Delta+\Id\quad\mbox{and}\quad A_{\rm rc}\coloneqq (a-1)\Id +b\cdot\nabla.
\end{equation}

It is straightforward to check that Assumptions~\ref{A:A} and~\ref{A:A1} are satisfied. Therefore, it remains to check the satisfiability of Assumptions~\ref{A:sens} and~\ref{A:act}, which we shall do in Sections~\ref{sS:parab-detect} and~\ref{sS:parab-stabil}. For this purpose,  we need  first to introduce appropriate sets of actuators and sensors in Section~\ref{sS:parab-actsens}.

\subsection{Actuators and sensors} \label{sS:parab-actsens}
For simplicity, we follow~\cite[Sect.~6]{Rod21-aut} by considering the same number of actuators and sensors, $M=S$.
Both will be taken as indicator functions of small rectangular subdomains as illustrated in
Fig.~\ref{fig:suppActSensGrid}, for a planar rectangle~$\Omega=\Omega^\times\in\bbR^2$. An analogue argument can be followed for a triangular domain and more generally for a convex polygonal/polyhedral domain~$\Omega\subset\bbR^d$, $d\in\bbN_+$; see~\cite[Rem.~2.8]{AzmiKunRod23}.

\begin{figure}[htbp]%
    \centering%
          \includegraphics[width=1\textwidth]{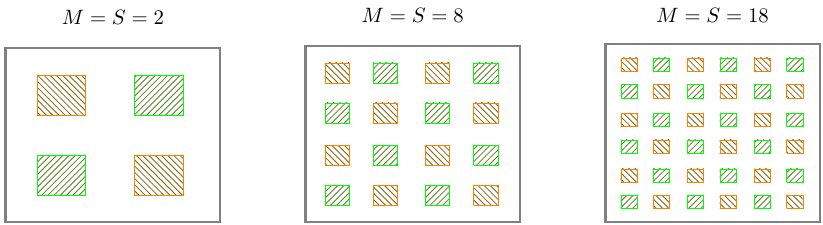}
          \caption{Supports of actuators (slash-$\slash$-lines) and sensors (backslash-$\backslash$-lines).} \label{fig:suppActSensGrid}
\end{figure}

 \begin{remark}
 The set of actuators and the set of sensors, as in Fig.~\ref{fig:suppActSensGrid}, cover each the same percentage of the domain,
independently of the
number of the actuators and sensors. 
 \end{remark}

\subsection{Detectability. Observer design} \label{sS:parab-detect}
We show the satisfiability of Assumption~\ref{A:sens}.
For this purpose, we consider average-like sensors in the form
\begin{equation}\label{output-form-intro}
w=\clZ y,\qquad w_j\coloneqq(\indf_{\omega_j^S},z)_H=\int_{\omega_j^M}z(x)\rmd x,
\end{equation}
where~$\indf_{\omega_j^S}$ is the indicator function of a rectangular subdomains
\begin{align}\notag
\omega_j^S\coloneqq\bigtimes_{n=1}^d(l_{(S,j,n,1)},l_{(S,j,n,2)})\subseteq\Omega, \qquad l_n=l_{(S,j,n,2)}-l_{(S,j,n,1)}>0.
 \end{align}
as in Fig.~\ref{fig:suppActSensGrid}, $S\in\{2s^d\mid s\in\bbN_+\}$. Let us denote the set of linearly independent sensors by
 \begin{align}
W_S\coloneqq \{\indf_{\omega_j^S}\mid 1\le j\le S\},\qquad \clW_S\coloneqq\linspan W_S,\qquad\dim\clW_S = S.\notag
 \end{align}

Let~$P_{F}$ denote the orthogonal projection in~$H$ onto~$F$.
The following result is a consequence of the result in~\cite[Thm.~3.1]{KunRodWal21}. It shows, in particular, the satisfiability of Assumption~\ref{A:sens}, with the
output injection operator as~$\fkL=\fkI_S\coloneqq-\lambda P_{\clW_{S}}\bfZ^{W_S}$.
\begin{lemma}\label{L:detect0}
Let~$\mu_1>0$. With the sensors localized as in Fig.~\ref{fig:suppActSensGrid}, for~$S$ and~$\lambda$ large enough, the system
 \begin{align}\notag
 &\hspace*{-.5em}\dot z +Az+A_{\rm rc}z
 =-\lambda P_{\clW_{S}}\bfZ^{W_S}\clZ z
,\qquad z(0)= z_0,
 \end{align}
is exponentially stable with rate~$\mu_1$.   The operator~$\bfZ^{W_S}\colon\bbR^{S}\to\clW_{S}$ is defined by
\begin{align}\notag
\bfZ^{W_S}\bfz 
\coloneqq\sum\limits_{j=1}^{S}
\left([\clV_{S}]^{-1} \bfz\right)_j\indf_{\omega_j^S},\qquad \bfz\in\bbR^{S},
\end{align}
 where~$[\clV_{S}]\in\bbR^{S\times S}$ is the matrix with entries in the~$i$-th row and~$j$-th column given by~$[\clV_{S}]_{(i,j)}=(\indf_{\omega_i^S},\indf_{\omega_j^S})_H.
$
\end{lemma}

\begin{remark}\label{Zout-orthproj}
The product/composition of the operators~$\bfZ^{W_S}$ and~$\clZ$ in Lemma~\ref{L:detect} coincide with the orthogonal projection~$P_{\clW_S}=\bfZ^{W_S}\clZ=P_{\clW_S}$; see~\cite[Eq.~(1.19b)]{Rod21-aut}. Thus, the  injected forcing in Lemma~\ref{L:detect0} coincides with the scaled orthogonal projection of the estimate error onto the space spanned by the sensors, $-\lambda P_{\clW_{S}}\bfZ^{W_S}\clZ_S z=-\lambda P_{\clW_{S}}z$.
\end{remark}

\subsection{Stabilizability and state-penalization.} \label{sS:parab-stabil}
We show the satisfiability of Assumption~\ref{A:act}. The existence of a stabilizing feedback control input follows again again by~\cite[Thm.~3.1]{KunRodWal21} for actuators as in Fig.~\ref{fig:suppActSensGrid} with large enough~$M\in\{2m^d\mid m\in\bbN_+\}$ and with the input feedback operator~$\fkL=\fkK\coloneqq-\lambda (U_M^\diamond)^{-1}P_{\clU_M}$, where
$(U_M^\diamond)^{-1}$ is the inverse of the isomorphism~$U_M^\diamond\colon \bbR^M\to\clU_M$.
 That is, we have that the pair~$(-A-A_{\rm rc},U_M^\diamond)$ is stabilizable.

It remains to give examples of state-penalization operators~$Q$, to be taken in the cost functional, so that~$(-A-A_{\rm rc},Q)$ is detectable.

\subsubsection{The cases~$Q\in\{\Id,A^\frac12\}$ with~$H_Q=H$}\label{ssS:exQ-id}
In these cases the detectability of~$(-A-A_{\rm rc},Q)$ can be concluded by taking~$\fkL=-\lambda\Id$, with~$\lambda>0$ large enough.

\subsubsection{The case~$Q=\clZ$ with~$H_Q=\bbR^S$.}\label{ssS:exQ-sens} 
In applications it may be convenient to take~$Q$ with finite-dimensional range, for example, to speed  computations up. Of course, we can take~$Q=\clZ$ once we know that~$(-A-A_{\rm rc},\clZ)$ is detectable.

\subsubsection{The case~$Q= P_{\clE_{N_Q}^\rmf}$ and~$H_Q=\clE_{N_Q}^\rmf\subset H$}\label{ssS:exQ-eig}
Here~$\clE_{N_Q}^\rmf$ is  the space spanned by ``the'' first~${N_Q}$ eigenfunctions of~$A$. The detectability of~$(-A-A_{\rm rc},Q)$ can be concluded for large enough~$N_Q$, by taking~$\fkL=-\lambda\Id$, with~$\lambda>0$ large enough.

 \section{Numerical simulations}\label{S:simul}
 We present results of simulations concerning  the coupled system~\eqref{sys-coupled} in the concrete setting of  scalar parabolic equations as follows
\begin{align*}
 &\tfrac{\p}{\p t} y_{\rm orh} -(\nu\Delta-\Id) y_{\rm orh}+(a-1)y_{\rm orh} +b\cdot\nabla y_{\rm orh}
 =U_{M}^\diamond u_{\rm orh},\\
&\tfrac{\p}{\p t} \widehat y -(\nu\Delta-\Id) \widehat y+(a-1)\widehat y +b\cdot\nabla \widehat y
=U_{M}^\diamond u_{\rm orh}+\fkI_S(\clZ\widehat y-\clZ    y_{\rm orh}),\\
  &\fkT \widehat y\rest{\Gamma}=0=\fkT y_{\rm orh}\rest{\Gamma},\quad
(y_{\rm orh}(0),\widehat y(0))=(y_0,\widehat y_0).
       \end{align*}

 \subsection{Test data}\label{sS:simul-data}
In the simulations we took the unit square~$\Omega=(0,1)\times(0,1)$ in~$\bbR^2$ as  spatial domain. Further, we took
\begin{align}
&\nu=0.1;\qquad &&a(x,t)=-\tfrac12+x_1 -|\sin(6t+x_1)|_\bbR;\notag\\
&\fkT=\bfn\cdot\nabla;\qquad &&b(x,t)=\begin{bmatrix}x_1+x_2\\|\cos(6t)x_1x_2|_\bbR\end{bmatrix},\notag \end{align}
and, as initial conditions we have taken
\begin{align}\notag
y_{\rm orh}(0)=y_{\rm orh}(x,0)=y_0(x)=1-2\cos(\pi x_1);\notag\\
\widehat y(0)=\widehat y(x,0)=\widehat y_0(x)=W_S^\diamond[\clV_{S}]^{-1}\clZ_S y_0(x).\notag
\end{align}
That is, we propose to take the guess~$\widehat y(x,0)=P_{\clW_S} y_0(x)$, using the information of the output~$\clZ_Sy_0(x)$ availabe at time~$t=0$. Here~$[\clV_{S}]$ is the matrix as in~Lemma~\ref{L:detect0}.

Finally,  $S=8$~sensors and $M=8$~actuators were taken as the indicator functions of the subdomains in a chessboard pattern as illustrated in Fig.~\ref{fig:suppActSensGrid}. The locations are also shown in Fig.~\ref{fig:mesh}, \begin{figure}[htbp]%
    \centering%
    \subfigure[Reference mesh~$\clT_0$.\label{fig:mesh0}]
    {\includegraphics[width=.48\textwidth]{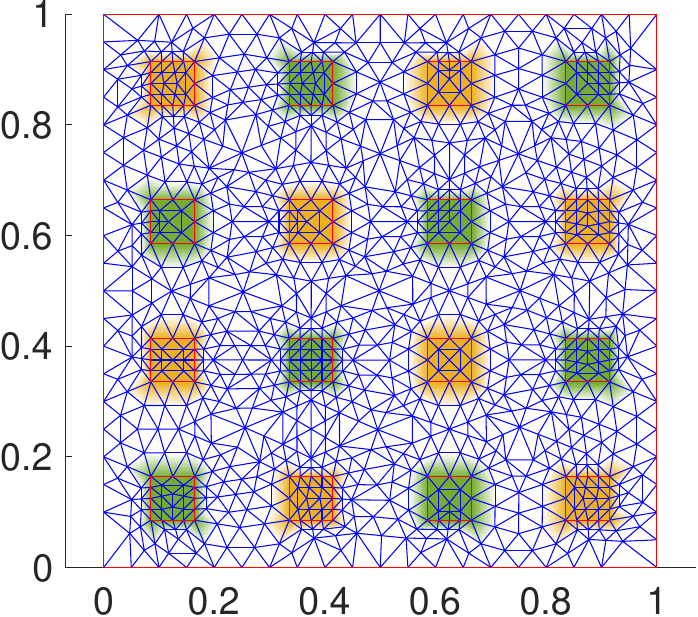}}\quad
    \subfigure[Regularly refined mesh~$\clT_1$.\label{fig:mesh1}]
         {\includegraphics[width=.48\textwidth]{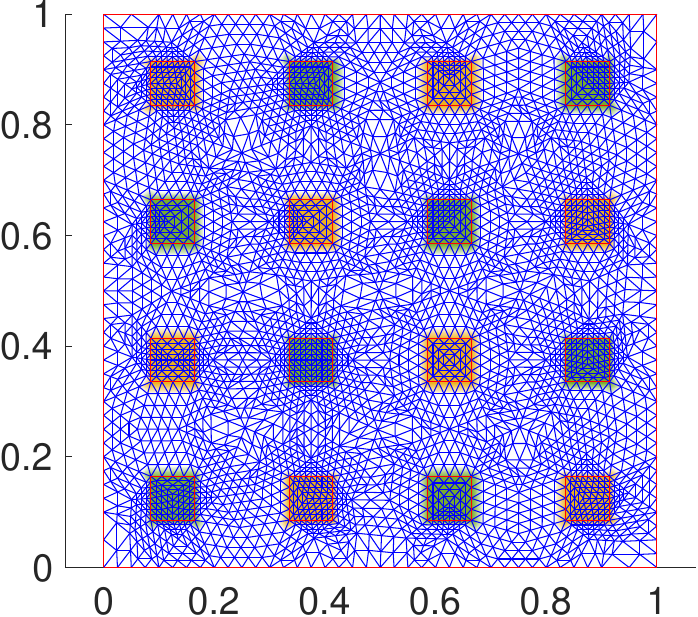}} 
         \caption{Spatial triangulations and locations of actuators and sensors.}%
     \label{fig:mesh}%
\end{figure}
together with the reference mesh/triangulation~$\clT_0$ used in the computations, shown in Fig.~\ref{fig:mesh0}. We show also the regular refinement~$\clT_1$ of the mesh~$\clT_0$ in Fig.~\ref{fig:mesh1}. These computations are based in a finite-element spatial discretization with piecewise-linear (hat) functions. In time a Crank--Nicolson discretization has been taken. The time-step was taken as~$t_0^{\rm step}=4\cdot10^{-4}$. Therefore, the simulations will be performed in the spatio-temporal discretization pairs as 
\begin{equation}\label{def.mesh}
\fkM_{\rm rf}\coloneqq\begin{cases}
(\clT_0,t_0^{\rm step}),&\mbox{for }{\rm rf}=0;\\
(\clT_1,\tfrac12 t_0^{\rm step}),&\mbox{for }{\rm rf}=1.
\end{cases}
\end{equation}

The state penalization in the cost functional~\eqref{cost-T} have been taken as
\begin{equation}\notag
Q=\sqrt{800}P_{\clE_{30}^\rmf}\in\clL(H,\clE_{30}^\rmf),
\end{equation}
where~$P_{\clE_{30}^\rmf}$ is the orthogonal projection operator onto the linear span~$\clE_{30}^\rmf\subset H$ of ``the'' first~$30$ eigenfunctions of the Neumann Laplacian. These eigenfunctions have been found through numerical computations.

We also look at the maximal squeezing factor~$\overline\theta$ obtained at concatenation times~$t_n$ given by Algorithm~\ref{alg:rhc+obs}, given by
\begin{equation}\notag
\overline\theta=\max\left\{\tfrac{\norm{y_T^*(t_{n+1})}{H}}{\norm{y_T^*(t_{n})}{H}}\mid t_n\in[0,T_\infty) \mbox{ is a concatenation time}\right\}.
\end{equation}
We stopped the Algorithm computations when the desired squeezing property
\begin{equation}\label{sqzOK}
\theta_n\coloneqq\tfrac{\norm{y_T^*(t_{n+1})}{H}}{\norm{y_T^*(t_{n})}{H}}<1
 \end{equation}
has been violated either in~$10$ consecutive  receding horizon  (RH) intervals~$\clI_n\coloneqq (t_{n},t_{n+1})$ or in~$50$ of such intervals in total.    These stopping criteria suggest the failure of the squeezing property for the RHC, which could stem either from the lack of theoretical stabilizability/detectability properties or simply from the accuracy of the computations of the optimal control problems, which are solved up to some chosen small tolerance value. Of course, the accuracy of computations is also subject to machine precision.

 \subsection{Solving the finite time-horizon  optimal control problems}\label{sS:simul-FTH-comput}
We solved the finite time-horizon open-loop optimal control problems iteratively down to a tolerance pair~${\rm Tol}=({\rm Tol1},{\rm Tol2})\in\bbR_+^2$, namely, down to the satisfiability of the conditions
\begin{subequations}\label{biTol}
\begin{align}
\max\left\{\norm{\delta^u_{k}}{L^2(t_n,t_n+T_{\rm rh})}^2,\norm{\delta^G_k}{L^2(t_n,t_n+T_{\rm rh})}^2\right\}&\le{\rm Tol1},\\
\norm{G(u_{k},p_{k})}{L^2(t_n,t_n+T_{\rm rh})}^2&\le{\rm Tol2},
\end{align}
\end{subequations}
where~$\delta^u_{k}$, $\delta^G_k$, and~$G$ stand for the differences as follows
\begin{align}
&\delta^u_{k}\coloneqq u_{k+1}^*-u_{k}^*,\qquad \delta^G_k\coloneqq G(u_{k+1}^*,p_{k+1})-G(u_{k}^*,p_{k}),\notag\\
&G(u,p)\coloneqq u-(U_M^\diamond)^*p.\notag
\end{align}
Here~$p_{k}$ is the iteration of the adjoint state solving
\begin{align}
 &\tfrac{\p}{\p t} p_{k} +(\nu\Delta-\Id) p_{k}-(a-1)p_{k} -b\cdot\nabla p_{k}
 =-Q^*Q y_{k}^*,\qquad&& p_{k}(t_n+T_{\rm rh})=0,
 \intertext{where the iterated optimal state solves}
 &\tfrac{\p}{\p t}y_{k}^*-(\nu\Delta-\Id)y_{k}^*+(a-1)y_{k}^* +b\cdot\nabla y_{k}^*
 =U_M^\diamond u_{k}^*,\qquad&& y_{k}^* (t_n)=\widehat y(t_n),
       \end{align}
and  both systems are solved under the considered homogeneous Neumann boundary conditions~$\bfn\cdot\nabla y_{k}^*\rest{\Gamma}=0=\bfn\cdot\nabla p_{k}\rest{\Gamma}$. That is, in the figures below~$y^*=y_{T_{\rm rh}}^*$ stands for the ``limit'' of such sequence~$y_{k}^*=y_{T_{\rm rh},k}^*$; for example, see Fig~\ref{fig:1stepOC}. Similarly, $u_{\rm orh}=u_{T_{\rm rh}}^*$ stands for the limit of such sequence~$u_{k}^*=u_{T_{\rm rh},k}^*$; for example, see Fig~\ref{fig:CL-Tmax4-tn}. The control input is updated as
\begin{align}
&u_{k+1}^*\coloneqq u_{k}^*-s^{\rm aBB}_{k}G(u^*_{k},p_{k}),\notag
\end{align}
where the~$s^{\rm aBB}_{k}$ are the  alternating Barzilai--Borwein   stepsizes  as in~\cite[Sect.~5]{AzmiKunisch20}.
The computations have been done up to the dynamic tolerance value~${\rm Tol}_n=({\rm Tol1}_n, {\rm Tol2}_n)$ as
\begin{subequations}\label{dynTol}
\begin{align}
{\rm Tol1}_n\coloneqq \max\{{\rm Tol_{\rm low}}(1),\min\{{\rm Tol_{\rm up}}(1),{\rm \widehat Tol_{n}}\}\},\\
{\rm Tol2}_n\coloneqq \max\{{\rm Tol_{\rm low}}(2),\min\{{\rm Tol_{\rm up}}(2),{\rm \widehat Tol_{n}}\}\},
\end{align}
where
\begin{align}
{\rm \widehat Tol_{n}}\coloneqq 10^{-2}\norm{\widehat y(t_n)}{H}^2(T_{\rm rh}+1)^{-1}.
\end{align}
\end{subequations}

For the reference minimal~${\rm Tol_{\rm low}}$ and maximal~${\rm Tol_{\rm up}}$ tolerance pairs we have taken
\begin{equation}\label{tolsmall}
{\rm Tol_{\rm low}}=(10^{-28},10^{-14})\quad\mbox{and}\quad {\rm Tol_{\rm up}}=(10^{-4},10^{-2}).
\end{equation}

 \subsection{Instability of the free dynamics and the output based optimal control}\label{sS:simul-6x6}
First of all we mention that the free/uncontrolled dynamics is unstable. This fact can be observed in Fig.~\ref{fig:free}, which shows the evolution of the norm of the state~$y_{\rm free}$ of the uncontrolled system
as follows
\begin{align}\label{sys-yfree} 
 &\tfrac{\p}{\p t} y_{\rm free} -(\nu\Delta-\Id) y_{\rm free}+(a-1)y_{\rm free} +b\cdot\nabla y_{\rm free}
 =0,\qquad 
y_{\rm free}(0)=y_0,
       \end{align}
 under the considered homogeneous Neumann boundary conditions~$\bfn\cdot\nabla y_{\rm free}\rest{\Gamma}=0$.
\begin{figure}[htbp]%
    \centering%
    \subfigure[Uncontrolled state.\label{fig:free}]
    {\includegraphics[width=.48\textwidth]{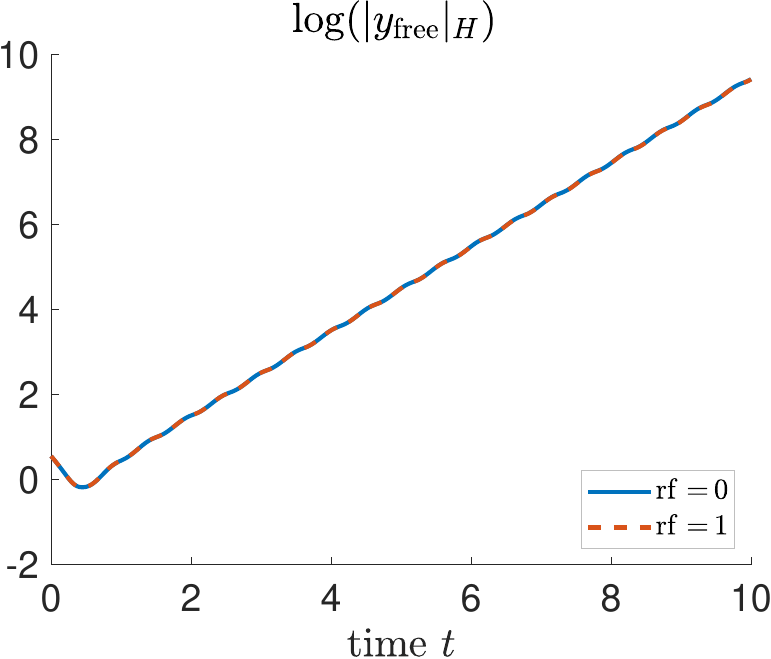}}\quad
    \subfigure[With estimated optimal control for $T_{\rm rh}=1$.\label{fig:1stepOC}]
         {\includegraphics[width=.48\textwidth]{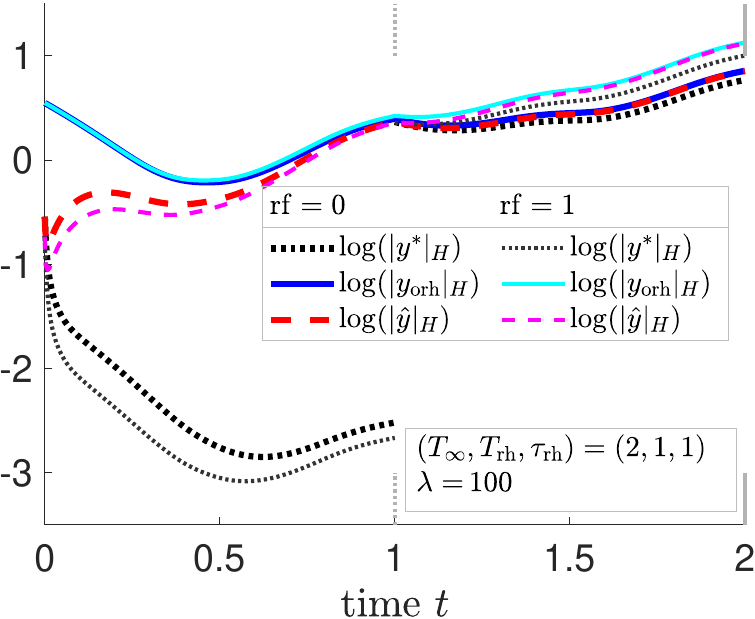}} 
         \caption{Free dynamics and estimate based optimal control.}%
     \label{fig:freeOC1}%
\end{figure}
Indeed, we can see that the norm of the state~$y_{\rm free}(t)$ of the uncontrolled system increases exponentially as time increases.

In Fig.~\ref{fig:free} we show the state computed in the reference mesh~$\fkM_0$ and in its refinement~$\fkM_1$; see~\eqref{def.mesh}. Visually, we note that the two solutions do coincide.
Next, in Fig.~\ref{fig:1stepOC}, we show  the solution corresponding to the guess/estimate based optimal control, again computed in the meshes~$\fkM_0$ and~$\fkM_1$. Now, visually, the solutions do not coincide, but the qualitative behavior is already captured by the solution computed in the coarsest mesh.
Since the computations in the coarsest mesh~$\fkM_0$ are considerably faster, hereafter, all the simulations correspond to the mesh~$\fkM_0$. 

 \subsection{ORHC stabilizing performance for several prediction horizon}\label{sS:simul-smallT_infty}
In this section, we consider the computational time~$T_\infty=4$ and the minimal sampling time~$\tau=\tau_{\rm rh}=0.1$. In Fig.~\ref{fig:ORHC-Tmax4-Trh}
\begin{figure}[htbp]%
    \centering%
      \subfigure[Prediction horizon $T_{\rm rh}=0.5$.\label{fig:ORHC-Tmax4-Trh0.5}]
    {\includegraphics[width=.48\textwidth]{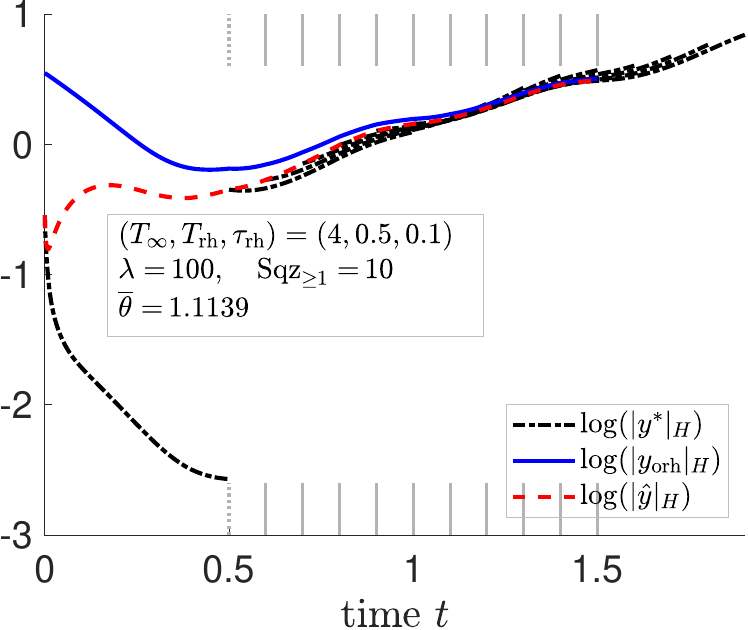}}\quad
     \subfigure[Prediction horizon $T_{\rm rh}=1$.\label{fig:ORHC-Tmax4-Trh1}]
    {\includegraphics[width=.48\textwidth]{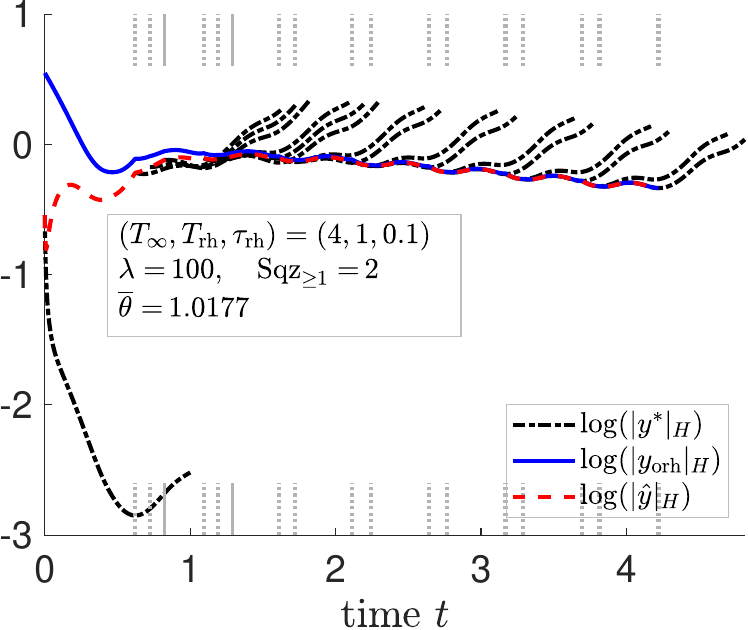}}\\
   \subfigure[Prediction horizon $T_{\rm rh}=1.5$.\label{fig:ORHC-Tmax4-Trh1.5}]
    {\includegraphics[width=.48\textwidth]{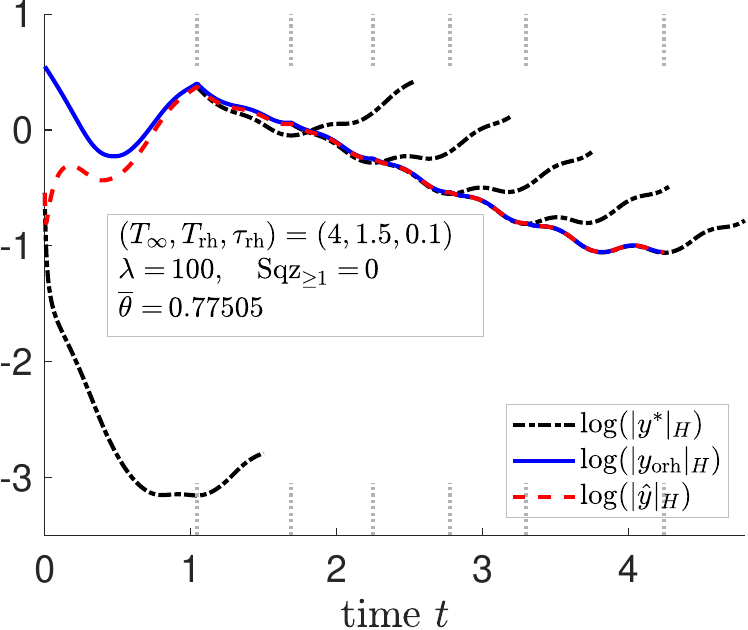}}\quad
    \subfigure[Prediction horizon $T_{\rm rh}=2$.\label{fig:ORHC-Tmax4-Trh2}]
    {\includegraphics[width=.48\textwidth]{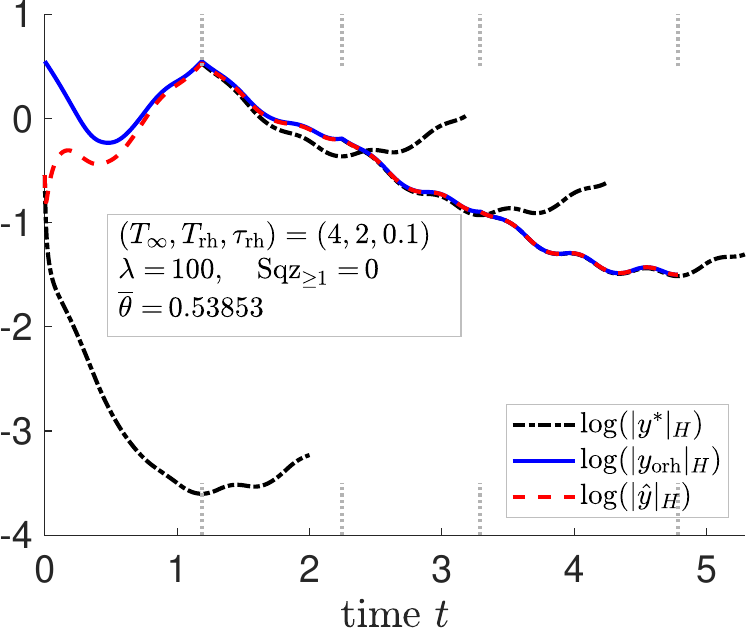}}\\
    \subfigure[Prediction horizon $T_{\rm rh}=2.5$.\label{fig:ORHC-Tmax4-Trh2.5}]
    {\includegraphics[width=.48\textwidth]{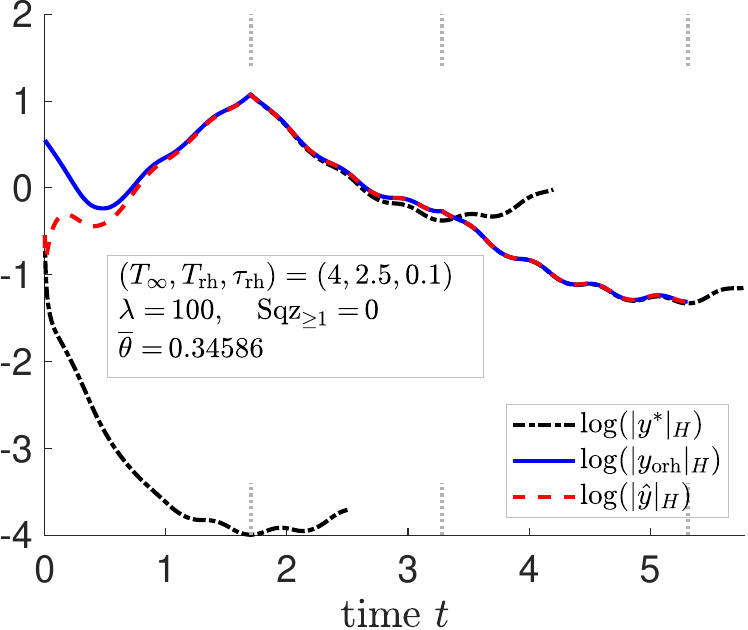}}\quad
       \subfigure[Prediction horizon $T_{\rm rh}=3$.\label{fig:ORHC-Tmax4-Trh3}]
    {\includegraphics[width=.48\textwidth]{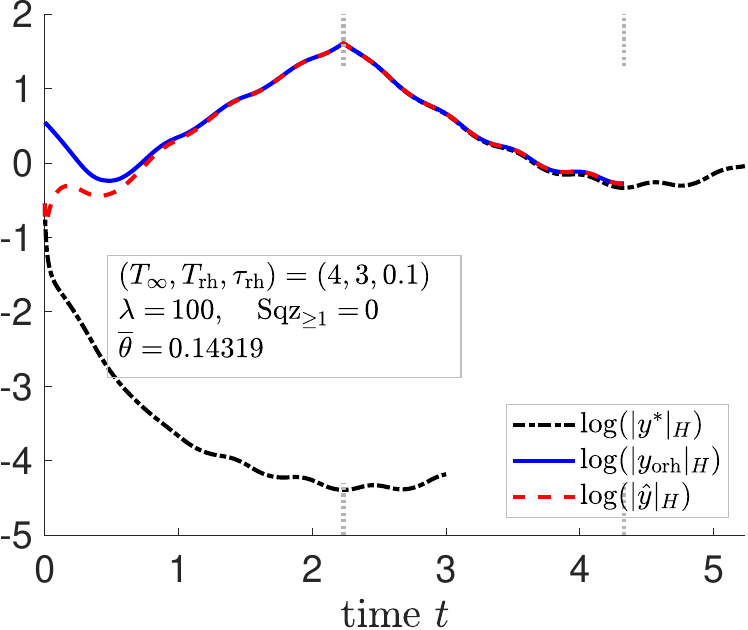}}
        \caption{ORHC performance for several prediction horizons~$T_{\rm rh}$.}%
     \label{fig:ORHC-Tmax4-Trh}%
\end{figure}
 we see the results obtained for several values taken for the prediction horizon~$T=T_{\rm rh}$. In the case~$T_{\rm rh}=0.5$ we can see, in~Fig.~\ref{fig:ORHC-Tmax4-Trh0.5},
that the norm of the controlled state~$y_{\rm orh}$ is increasing for large~$t$ which leads to the conclusion that the  ORHC is likely not stabilizing. This is also supported by the fact that the squeezing property~\eqref{sqzOK} has been violated  in the ten consecutive RH time intervals~$(t_n,t_{n+1})$, $1\le n\le 10$,   while being satisfied only in the  first interval~$(t_0,t_1)$. The vertical segment lines at the top and bottom of the figures are located at the concatenation time  instants~$t_{n+1}$, $n\ge0$; a dotted segment indicates that the squeezing property~\eqref{sqzOK} was satisfied in the  RH interval~$(t_n,t_{n+1})$ and a continuous segment indicates that the squeezing property~\eqref{sqzOK} was violated in~$(t_n,t_{n+1})$. In the figures, the total number of violations is denoted by
\begin{equation}\notag
{\rm Sqz_{\ge1}}\coloneqq \#\{t_n\mid \mbox{ the property~\eqref{sqzOK} is not satisfied}\}.
\end{equation}

In the case~$T_{\rm rh}=1$ the squeezing property~\eqref{sqzOK} has been violated only in two RH time intervals, namely~$(t_2,t_{3})$ and~$(t_5,t_{6})$ as we can see in~Fig.~\ref{fig:ORHC-Tmax4-Trh1}. The results also suggest that the ORHC is stabilizing with some (small) exponential rate. These results suggest that the squeezing property~\eqref{sqzOK} may be violated provided that it holds in subsequent time intervals. This point shall be shortly revisited in Sect.~\ref{sS:finrmks-deltaT}.

For larger~$T_{\rm rh}\in\{1.5,2,2.5,3\}$ the squeezing property~\eqref{sqzOK} held true in all the  RH time intervals~$(t_n,t_{n+1})$ as we can see in Figs.~\ref{fig:ORHC-Tmax4-Trh1.5}--\ref{fig:ORHC-Tmax4-Trh3}.

We see that for larger~$T_{\rm rh}$ we will obtain larger lengths~$t_{n+1}-t_n$ for the RH time intervals, thus we have to solve a smaller number of optimization problems. However, solving the optimal controls in larger time intervals is more time consuming. So, the choice of~$T_{\rm rh}$ is nontrivial and may depend on the concrete application: we know that~$T_{\rm rh}$ must be large enough (for a given suitable state penalization parameter~$Q$; see Lemma~\ref{L:optTsqueez}), but it may be important to choose~$T_{\rm rh}$ not too large so that, for example, we can perform the computations faster (e.g., in real time).

 \subsection{ORHC stabilizing performance for larger computation time}\label{sS:simul-largeTinfty}
 In this section, we consider the computational time~$T_\infty=100$. We fix~$(T_{\rm rh},\tau_{\rm rh})=(2,0.1)$. Since~$T_\infty$ is large we do not save the solution at all discrete times, but rather save it at the  key concatenation time instants~$t_n$ only. For illustration, we recompute firstly for the case~$T_\infty=4$ and show the results in Fig.~\ref{fig:ORHC-Tmax4-tn-lam100}, which correspond to the results in Fig.~\ref{fig:ORHC-Tmax4-Trh2}.
 In this section we also consider the case where the provided state estimate by the observer converges to the real (controlled) state at a smaller exponential rate. For this we simply decrease the observer gain down to~$\lambda=19$. The results are shown in Fig.~\ref{fig:ORHC-Tmax4-tn-lam19}.
 \begin{figure}[htbp]%
    \centering%
      \subfigure[Observer scalar gain $\lambda=100$.\label{fig:ORHC-Tmax4-tn-lam100}]
    {\includegraphics[width=.48\textwidth]{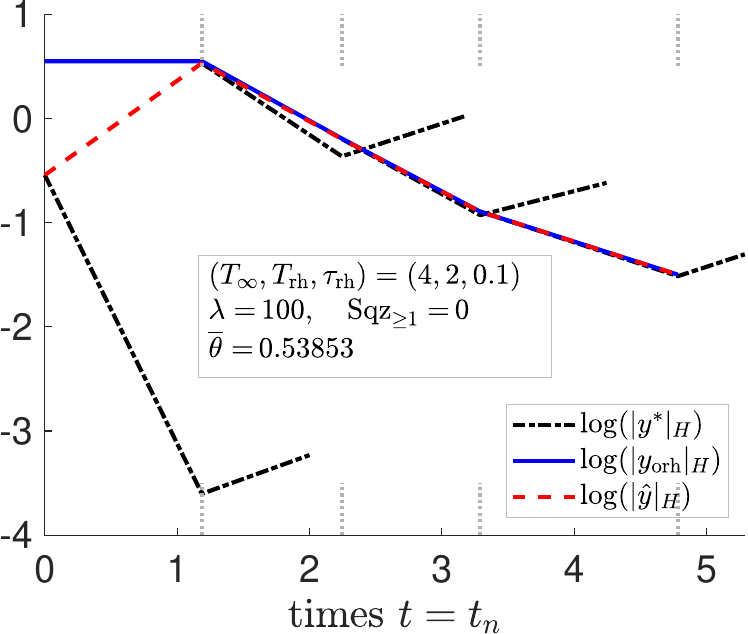}}\quad
     \subfigure[Observer scalar gain $\lambda=19$.\label{fig:ORHC-Tmax4-tn-lam19}]
    {\includegraphics[width=.48\textwidth]{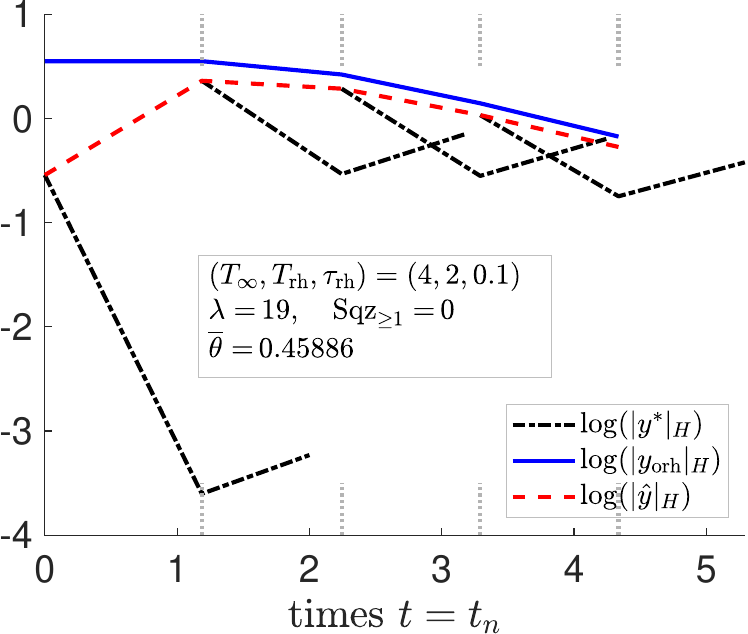}}
         \caption{ORHC performance (at times~$t_n$). $T_\infty=4$.}%
     \label{fig:ORHC-Tmax4-tn}%
\end{figure}

In order to better see the performance of the observer we plot the evolution of the norm of the state estimate error~$\widehat y-y_{\rm orh}$  in Fig.~\ref{fig:CL-Tmax4-tn}.
\begin{figure}[htbp]%
    \centering%
      \subfigure[Observer scalar gain $\lambda=100$.\label{fig:CL-Tmax4-tn-lam100}]
    {\includegraphics[width=.48\textwidth]{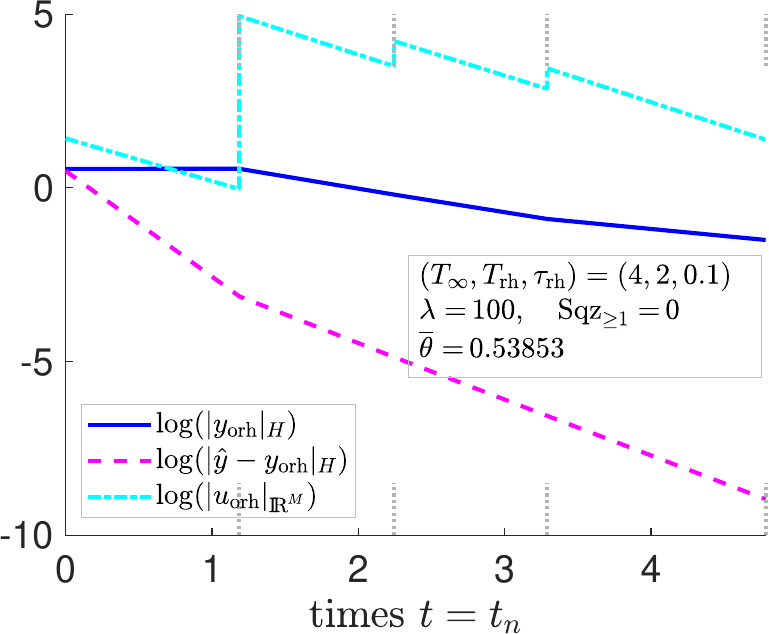}}\quad
     \subfigure[Observer scalar gain $\lambda=19$.\label{fig:CL-Tmax4-tn-lam19}]
    {\includegraphics[width=.48\textwidth]{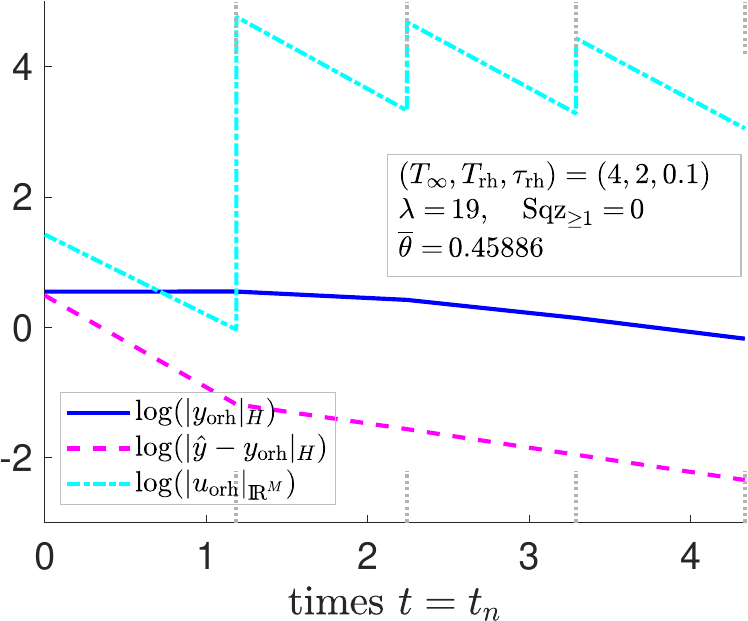}}
         \caption{ORHC and observer performance (at times~$t_n$). $T_\infty=4$.}%
     \label{fig:CL-Tmax4-tn}%
\end{figure}
In the same figure we also see the evolution of the norm of the  ORHC  input~$u_{\rm orh}$. In particular, we note that the control is discontinuous at the concatenation points; such discontinuities are expected (also in full-state--based RHC strategies; cf.~\cite[Fig.~5]{AzmiKunisch19}). To plot/underline these discontinuities we have saved the control input not only at the concatenation times~$t_n$,  but also at the preceding discrete times $t_{n+1}-t_0^{\rm step}$, $n\ge0$.

To have a more complete picture of the performance of the proposed strategy, we have run the simulations up to $T_\infty=100$. The results are shown in Fig.~\ref{fig:CL-Tmax100-tn}.
\begin{figure}[htbp]%
    \centering%
     \subfigure[Observer scalar gain $\lambda=100$.]
   {\includegraphics[width=.48\textwidth]{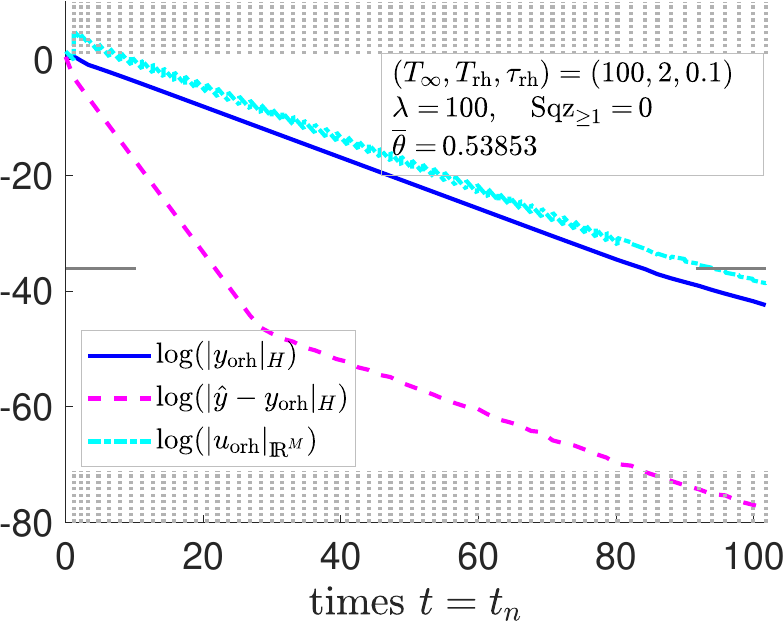}}\quad
     \subfigure[Observer scalar gain $\lambda=19$.]
    {\includegraphics[width=.48\textwidth]{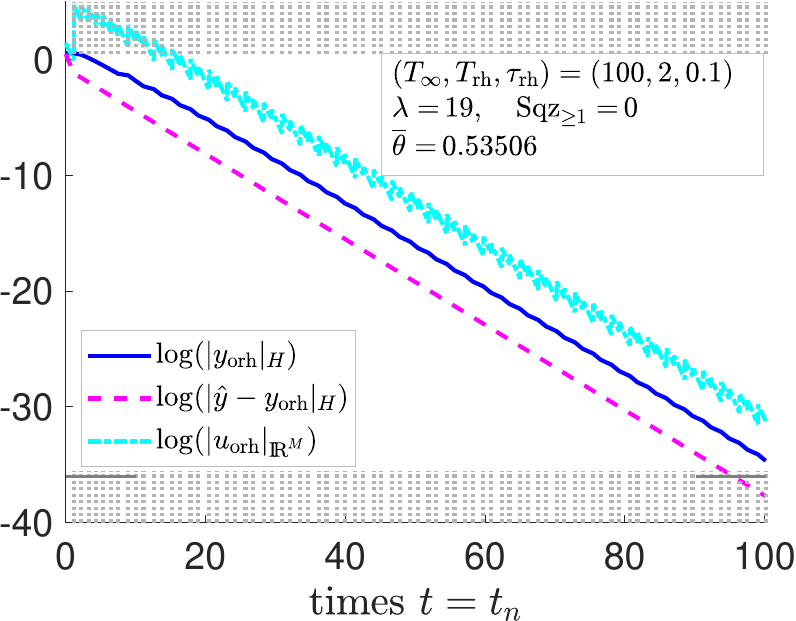}}
          \caption{ORHC  and observer performance (plot at times~$t_n$). $T_\infty=100$.}%
     \label{fig:CL-Tmax100-tn}%
\end{figure}
 We confirm that for~$\lambda=19$ the estimate~$\widehat y(t)$ provided by the observer converges to~$y_{\rm orh}(t)$ slower than for~$\lambda=100$. In either case, we confirm the stabilizing property of the  ORHC.

The computations have been performed in Matlab with machine precision~${\tt eps}\approx2.204\cdot 10^{-16}$ (see the horizontal line segments located at $\log({\tt eps})\approx-36.0437$ in the vertical axis).

We see that for time~$t>20$  the error estimate~$z(t)=\widehat y(t)-y_{\rm orh}(t)$ is below machine precision in the case~$\lambda=100$. Thus, in this case the observed  ORHC performance, for~$t>20$, is the one we would obtain with the corresponding full-state--based RHC (up to machine precision state estimation errors). Machine precision and associated round-off errors could be the reason why the detecting rate of the observer clearly deteriorates for time~$t>30$ in the case~$\lambda=100$.

In Fig.~\ref{fig:CL-Tmax100-tn-lam19100} 
\begin{figure}[htbp]%
    \centering%
   {\includegraphics[width=.48\textwidth]{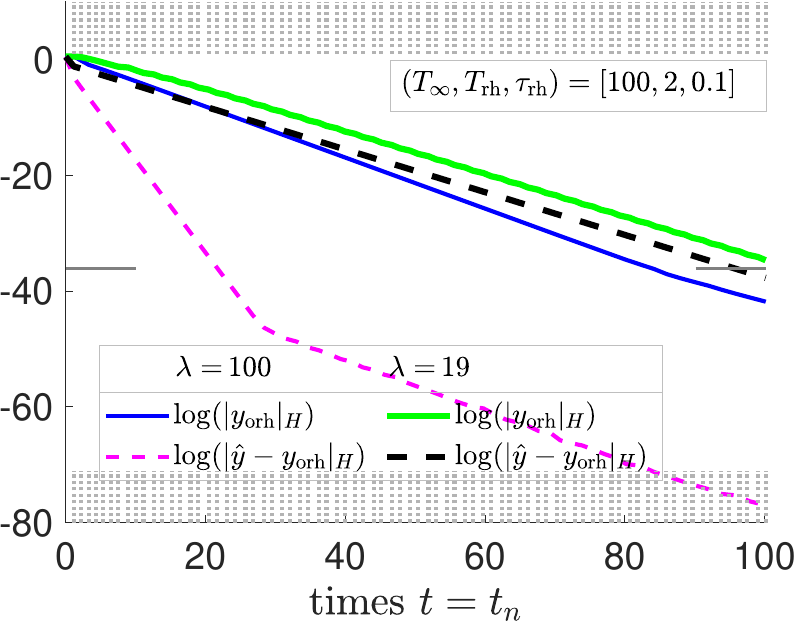}}\quad
          \caption{On the ORHC exponential decrease rates.}%
     \label{fig:CL-Tmax100-tn-lam19100}%
\end{figure}
we plot the corresponding controlled state states and error estimates together to better compare the decrease rates.
In the case~$\lambda=100$, where the observer is detecting/estimating with a relatively large exponential rate, we can see that the output-based  RHC provides a larger stabilization rate for the controlled state~$y_{\rm orh}$. In particular, we observe that such stabilization rate achieved for the controlled state is, in either case, not larger than the detection/estimation rate provided by the observer.

 \subsection{On the accuracy of the solution of the FTH optimal control problems}\label{sS:tolFTH}
Solving the FTH open-loop optimal control problems can be a time consuming numerical task. We discuss briefly the case where the FTH open-loop optimal control problems are solved iteratively, as in Sect.~\ref{sS:simul-FTH-comput}, down to a larger tolerance, thus, with less accuracy. 
The previous simulations have been performed with the minimal tolerance pair~${\rm Tol_{\rm low}}=(10^{-28},10^{-14})$ as in~\eqref{tolsmall}. Now, we take the same maximal tolerance pair~${\rm Tol_{\rm up}}$ as in~\eqref{tolsmall}, but we take a larger minimal one as~${\rm Tol_{\rm low}}=(10^{-8},10^{-4})$. This leads to the 
analogue of Fig.~\ref{fig:CL-Tmax100-tn} shown in Fig.~\ref{fig:CL-Tmax100-tn-tollarge}.
\begin{figure}[htbp]%
    \centering%
     \subfigure[Observer scalar gain $\lambda=100$.\label{fig:CL-Tmax100-tn-tollarge-lam100}]
   {\includegraphics[width=.48\textwidth]{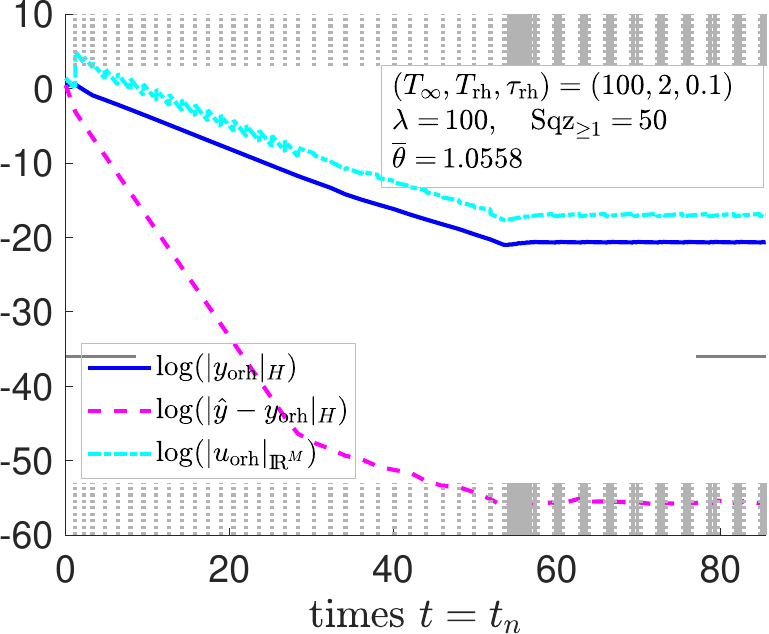}}\quad
     \subfigure[Observer scalar gain $\lambda=19$.\label{fig:CL-Tmax100-tn-tollarge-lam30}]
    {\includegraphics[width=.48\textwidth]{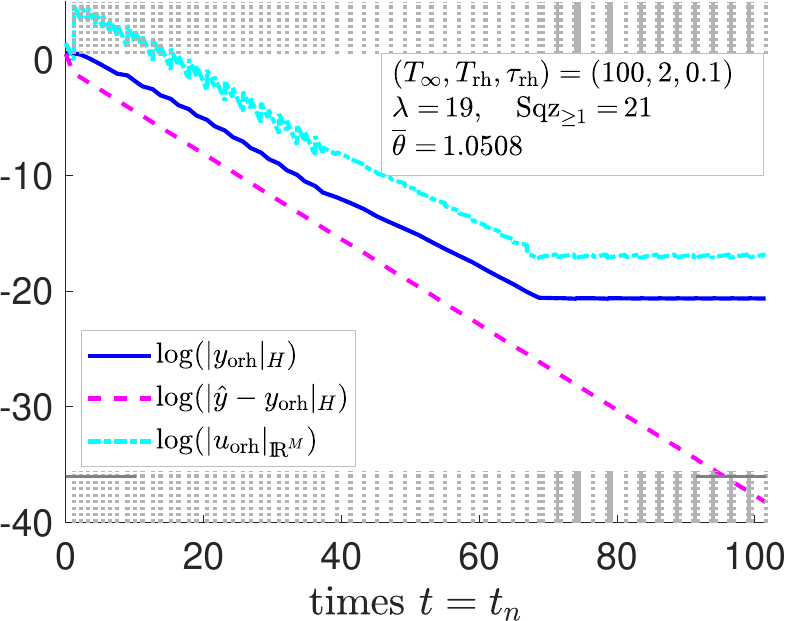}}
          \caption{Behavior for larger minimal tolerance pair~${\rm Tol_{\rm low}}=(10^{-8},10^{-4})$.}%
     \label{fig:CL-Tmax100-tn-tollarge}%
\end{figure}
That is, by solving the FTH  optimal control problems with less accuracy, we reach a stagnation-like behavior for the norms as shown in Fig.~\ref{fig:CL-Tmax100-tn-tollarge}. This is natural and expected if, for example, we look at the associated larger numerical errors as small (but, still relevant)  perturbations of the dynamics. Recall that, as mentioned in Sect.~\ref{sS:simul-data}, we have stopped the computations when the squeezing property
in~\eqref{sqzOK} 
has been violated either in~$10$ consecutive RH intervals~$\clI_n\coloneqq (t_{n},t_{n+1})$ or in a total of~$50$ intervals, the latter is the reason why the simulations were stopped before~$T_\infty=100$ in Fig.~\ref{fig:CL-Tmax100-tn-tollarge-lam100}. We see also that the stagnation behavior of the norms starts when the state norm reaches the value~$\rme^{-20}$ approximately, in both cases. This fact can be seen as a motivation to take a dynamic tolerance as in~\eqref{dynTol}, depending on the norm of the (available) state estimate~$\widehat y(t_n)$.

 \section{Conclusions}\label{S:finalremks}
We have shown the stability of the closed-loop system coupling a RHC framework with a Luenberger observer in the context of parabolic-like equations and we have presented and discussed results of corresponding simulations. The stability is guaranteed for appropriate sets of actuators and sensors, namely, for appropriate stabilizability and detectability properties. The approach depends on a triple~$(Q,\tau,T)$ at our disposal, which we can use to tune the proposed Algorithm~\ref{alg:rhc+obs}. Stability holds for a prediction horizon~$T$ large enough and for appropriate detectability properties of the state-penatization operator~$Q$.

\subsection{On the state-penalization operator~$Q$}\label{sS:finrmks-Q}
In the literature, the operator~$Q$, penalizing the state in the cost functional, is sometimes called ``observation'' operator; see~\cite[Part~I, Ch.~1, Sect.~3, Eqs.~(3.1) and~(3.3)]{BensoussanDaPratoDelfMitt07}. In this manuscript it is called ``state-penalization'' operator. This is done to avoid any possible confusion with the ``observer'' providing us with a state estimate. That is, the operator~$Q$ is simply seen as a tuning ``parameter'' in the cost functional; see~\cite[Sect.1.3.1]{RawlingsMayneDiehl19}. 
In other words, the choice of the operator~$Q$ is at our disposal, independently of the sets of sensors and actuators. It can be taken, for example, as one of those operators in Sects.~\ref{ssS:exQ-id}, \ref{ssS:exQ-sens}, or~\ref{ssS:exQ-eig}.
Concerning future work, it would be interesting to
 investigate how the tuning/choice of~$Q$ influences the performance of Algorithm~\ref{alg:rhc+obs} towards, for example, either maximizing the provided stabilization rate  or reducing the computation time (e.g., by guaranteeing a faster computation of the optimal FTH controls or by minimizing the value of the suitable  prediction horizon~$T$).
 
\subsection{On the ORHC pair~$(\tau,T)$}\label{sS:finrmks-deltaT}
The prediction horizon~$T$ must be large enough so that the concatenation time instants~$t_n$, computed online, can be selected so that the norm of the optimal state at time~$t_n$ is squeezed in comparison with the norm at initial time~$ {t_{\rm in}} =t_{n-1}$; see Algorithm~\ref{alg:rhc+obs} and Lemma~\ref{L:optTsqueez}.
The available literature, on RHC, considers mostly the case where the concatenation time instants are taken simply as~$n\tau$, with~$ \tau$ (sampling time)  chosen apriori (for larger~$T=T(\tau)$, if necessary).
It would be interesting to know whether we can take such concatenation steps in Algorithm~\ref{alg:rhc+obs}.
This is a nontrivial question. Further, it is not clear whether we can fix~$\tau$ arbitrarily, for example,   as in the results of~\cite{AzmiKunisch19,KunPfeiffer20}, which apply to the particular case where it is assumed that we have access to the entire state.  We also note the following: in~\cite[Thm.~2.6 and Rem.~2.7]{AzmiKunisch19} (addressing linear nonautonomous dynamics) the value of~$\tau$ is taken arbitrarily, but  in~\cite[Thm.~2.4 and Rem.~2.5]{KunPfeiffer20} (addressing nonlinear autonomous dynamics) this value is taken large enough, as~$\tau\ge\tau_0$ for some~$\tau_0>0$.

We used (asked for) the squeezing property~\eqref{sqzOK} in every  RH time interval~$(t_n,t_{n+1})$. In the infinite time-horizon setting~$T_\infty=\infty$, it is clear that~\eqref{sqzOK} may be violated in a finite number of such time intervals. Of course, it may also be violated in a countable number of such intervals, provided these violations are compensated by the squeezing properties in the remaining time intervals. It could be interesting to investigate and quantify the set of intervals where~\eqref{sqzOK} may be violated, but the exponential stability is still guaranteed for the proposed  ORHC  strategy.

\subsection{Robustness}\label{sS:finrmks-robust}
The theoretical and numerical investigation of   the robustness properties of the strategy,  specifically,  against sensor measurement errors is an interesting subject for future work due to the ubiquitous presence of such errors in real-world applications.  Other types of disturbances could also be investigated,   such as those arising from model uncertainties or noisy  external forces.

\medskip\noindent
{\bf Aknowlegments.}
S. Rodrigues gratefully acknowledges partial support from
the State of Upper Austria and the Austrian Science
Fund (FWF): P 33432-NBL.

\renewcommand*{\bibfont}{\normalfont\small}
{
\small
\bibliographystyle{plainurl}
\bibliography{RHC_Obs}

\begin{thebibliography}{10}

\bibitem{AzmiKunisch19}
B.~Azmi and K.~Kunisch.
\newblock A hybrid finite-dimensional {RHC} for stabilization of time-varying
  parabolic equations.
\newblock {\em SIAM J. Control Optim.}, 57(5):3496--3526, 2019.
\newblock \href {http://dx.doi.org/10.1137/19M1239787}
  {\path{doi:10.1137/19M1239787}}.

\bibitem{AzmiKunisch20}
B.~Azmi and K.~Kunisch.
\newblock Analysis of the {B}arzilai-{B}orwein step-sizes for problems in
  {H}ilbert spaces.
\newblock {\em J. Optim. Theory Appl.}, 185(3):819--844, 2020.
\newblock \href {http://dx.doi.org/10.1007/s10957-020-01677-y}
  {\path{doi:10.1007/s10957-020-01677-y}}.

\bibitem{AzmiKunRod23}
B.~Azmi, K.~Kunisch, and S.~S. Rodrigues.
\newblock Saturated feedback stabilizability to trajectories for the
  {S}chl\"{o}gl parabolic equation.
\newblock {\em IEEE Trans. Automat. Control}, 68(12):7089--7103, 2023.
\newblock \href {http://dx.doi.org/10.1109/tac.2023.3247511}
  {\path{doi:10.1109/tac.2023.3247511}}.

\bibitem{Bacciotti19}
A.~Bacciotti.
\newblock {\em Stability and Control of Linear Systems}.
\newblock Springer, 2019.
\newblock \href {http://dx.doi.org/10.1007/978-3-030-02405-5}
  {\path{doi:10.1007/978-3-030-02405-5}}.

\bibitem{Beckner75}
W.~Beckner.
\newblock Inequalities in {F}ourier analysis.
\newblock {\em Annals Math.}, 102(1):159--182, 1975.
\newblock \href {http://dx.doi.org/10.2307/1970980}
  {\path{doi:10.2307/1970980}}.

\bibitem{BensoussanDaPratoDelfMitt07}
A.~Bensoussan, G.~{Da Prato}, M.~C. Delfour, and S.~K. Mitter.
\newblock {\em Representation and control of infinite dimensional systems},
  volume~2.
\newblock Springer, 2007.
\newblock \href {http://dx.doi.org/10.1007/978-0-8176-4581-6}
  {\path{doi:10.1007/978-0-8176-4581-6}}.

\bibitem{BrascampLieb76}
H.~J. Brascamp and E.~H. Lieb.
\newblock Best constants in {Y}oung's inequality, its converse, and its
  generalization to more than three functions.
\newblock {\em Advances Math.}, 20(2):151--173, 1976.
\newblock \href {http://dx.doi.org/10.1016/0001-8708(76)90184-5}
  {\path{doi:10.1016/0001-8708(76)90184-5}}.

\bibitem{FranzeLucia15}
G.~Franze and W.~Lucia.
\newblock The obstacle avoidance motion planning problem for autonomous
  vehicles: A low-demanding receding horizon control scheme.
\newblock {\em Syst. Control Lett.}, 77:1 -- 10, 2015.
\newblock \href {http://dx.doi.org/10.1016/j.sysconle.2014.12.007}
  {\path{doi:10.1016/j.sysconle.2014.12.007}}.

\bibitem{GarciaPrettMorari89}
C.~E. Garcia, D.~M. Prett, and M.~Morari.
\newblock Model predictive control: Theory and practice -- a survey.
\newblock {\em Automatica J. IFAC}, 25(3):335--348, 1989.
\newblock \href {http://dx.doi.org/10.1016/0005-1098(89)90002-2}
  {\path{doi:10.1016/0005-1098(89)90002-2}}.

\bibitem{GruenePannek09}
L.~Gruene and J.~Pannek.
\newblock Practical {NMPC} suboptimality estimates along trajectories.
\newblock {\em Syst. Control Lett.}, 58:161--168, 2009.
\newblock \href {http://dx.doi.org/10.1016/j.sysconle.2008.10.012}
  {\path{doi:10.1016/j.sysconle.2008.10.012}}.

\bibitem{GruenePannek17}
L.~Gruene and J.~Pannek.
\newblock {\em Nonlinear Model Predictive Control: Theory and Algorithms}.
\newblock Springer, 2 edition, 2017.
\newblock \href {http://dx.doi.org/10.1007/978-3-319-46024-6}
  {\path{doi:10.1007/978-3-319-46024-6}}.

\bibitem{GrueneRantzer08}
L.~Gruene and A.~Rantzer.
\newblock On the infinite horizon performance of receding horizon controllers.
\newblock {\em IEEE Trans. Autom. Control}, 53(9):2100--2111, 2008.
\newblock \href {http://dx.doi.org/10.1109/TAC.2008.927799}
  {\path{doi:10.1109/TAC.2008.927799}}.

\bibitem{ItoKuni06}
Kazufumi Ito and Karl Kunisch.
\newblock Receding horizon control with incomplete observations.
\newblock {\em SIAM J. Control Optim.}, 45(1):207--225, 2006.
\newblock \href {http://dx.doi.org/10.1137/S0363012903437988}
  {\path{doi:10.1137/S0363012903437988}}.

\bibitem{KunPfeiffer20}
K.~Kunisch and L.~Pfeiffer.
\newblock The effect of the terminal penalty in receding horizon control for a
  class of stabilization problems.
\newblock {\em ESAIM: Control Optim. Calc. Var.}, 26:art.58, 2020.
\newblock \href {http://dx.doi.org/10.1051/cocv/2019037}
  {\path{doi:10.1051/cocv/2019037}}.

\bibitem{KunRodWal21}
K.~Kunisch, S.~S. Rodrigues, and D.~Walter.
\newblock Learning an optimal feedback operator semiglobally stabilizing
  semilinear parabolic equations.
\newblock {\em Appl. Math. Optim.}, 84(S1):277--318, 2021.
\newblock \href {http://dx.doi.org/10.1007/s00245-021-09769-5}
  {\path{doi:10.1007/s00245-021-09769-5}}.

\bibitem{KwonPearson77}
W.~Kwon and A.~Pearson.
\newblock A modified quadratic cost problem and feedback stabilization of a
  linear system.
\newblock {\em IEEE Trans. Autom. Control}, 22(5):838--842, 1977.
\newblock \href {http://dx.doi.org/10.1109/TAC.1977.1101619}
  {\path{doi:10.1109/TAC.1977.1101619}}.

\bibitem{LiYanShi17}
H.~Li, W.~Yan, and Y.~Shi.
\newblock A receding horizon stabilization approach to constrained nonholonomic
  systems in power form.
\newblock {\em Syst. Control Lett.}, 99:47--56, 2017.
\newblock \href {http://dx.doi.org/10.1016/j.sysconle.2016.11.005}
  {\path{doi:10.1016/j.sysconle.2016.11.005}}.

\bibitem{Longchamp83}
R.~Longchamp.
\newblock Singular perturbation analysis of a receding horizon controller.
\newblock {\em Automatica J. IFAC}, 19(3):303--308, 1983.
\newblock \href {http://dx.doi.org/10.1016/0005-1098(83)90108-5}
  {\path{doi:10.1016/0005-1098(83)90108-5}}.

\bibitem{Lunardi91}
A.~Lunardi.
\newblock Stabilizability of time-periodic parabolic equations.
\newblock {\em SIAM J. Control Optim.}, 29(4):810--828, 1991.
\newblock \href {http://dx.doi.org/10.1137/0329044}
  {\path{doi:10.1137/0329044}}.

\bibitem{MayneMichalska90}
D.~Q. Mayne and H.~Michalska.
\newblock Receding horizon control of nonlinear systems.
\newblock {\em IEEE Trans. Automat. Control}, 35(7):814--824, 1990.
\newblock \href {http://dx.doi.org/10.1109/9.57020}
  {\path{doi:10.1109/9.57020}}.

\bibitem{MayneRavicFindAllg09}
D.~Q. Mayne, S.~V. Rakovi{\'c}, R.~Findeisen, and F.~Allg{\"o}wer.
\newblock Robust output feedback model predictive control of constrained linear
  systems: time varying case.
\newblock {\em Automatica}, 45(9):2082--2087, 2009.
\newblock \href {http://dx.doi.org/10.1016/j.automatica.2009.05.009}
  {\path{doi:10.1016/j.automatica.2009.05.009}}.

\bibitem{MayneRawlRaoScok00}
D.~Q. Mayne, J.~B. Rawlings, C.~V. Rao, and P.~O.~M. Scokaert.
\newblock Constrained model predictive control: Stability and optimality.
\newblock {\em Automatica J. IFAC}, 36(6):789--814, 2000.
\newblock \href {http://dx.doi.org/10.1016/S0005-1098(99)00214-9}
  {\path{doi:10.1016/S0005-1098(99)00214-9}}.

\bibitem{RawlingsMayneDiehl19}
J.~B. Rawlings, D.~Q. Mayne, and M.~M. Diehl.
\newblock {\em Model Predictive Control: Theory, Computation, and Design}.
\newblock Nob Hill Publishing, 2 edition, 2019.
\newblock 1st printing.
\newblock URL:
  \url{http://www.nobhillpublishing.com/mpc-paperback/index-mpc.html}.

\bibitem{Rod23-eect}
S.~S. Rodrigues.
\newblock Stabilization of nonautonomous linear parabolic-like equations:
  oblique projections versus {Riccati} feedbacks.
\newblock {\em Evol. Equ. Control Theory}, 12(2):647--686, 2023.
\newblock \href {http://dx.doi.org/10.3934/eect.2022045}
  {\path{doi:10.3934/eect.2022045}}.

\bibitem{Rod21-aut}
S.S. Rodrigues.
\newblock Oblique projection output-based feedback stabilization of
  nonautonomous parabolic equations.
\newblock {\em Automatica J. IFAC}, 129:art109621, 2021.
\newblock \href {http://dx.doi.org/10.1016/j.automatica.2021.109621}
  {\path{doi:10.1016/j.automatica.2021.109621}}.

\bibitem{Rod23-scl}
S.S. Rodrigues.
\newblock Remarks on finite and infinite time-horizon optimal control problems.
\newblock {\em Syst. Control Lett.}, 172:art105441, 2023.
\newblock \href {http://dx.doi.org/10.1016/j.sysconle.2022.105441}
  {\path{doi:10.1016/j.sysconle.2022.105441}}.

\bibitem{VeldZua22}
D.~W.~M. Veldman and E.~Zuazua.
\newblock A framework for randomized time-splitting in linear-quadratic optimal
  control.
\newblock {\em Numer. Math.}, 151(2):495--549, 2022.
\newblock \href {http://dx.doi.org/10.1007/s00211-022-01290-3}
  {\path{doi:10.1007/s00211-022-01290-3}}.

\bibitem{VeldBorZua24}
D.W.~M. Veldman, A.~Borkowski, and E.~Zuazua.
\newblock Stability and convergence of a randomized model predictive control
  strategy.
\newblock {\em IEEE Transactions on Automatic Control}, pages 1--8, 2024.
\newblock \href {http://dx.doi.org/10.1109/TAC.2024.3375253}
  {\path{doi:10.1109/TAC.2024.3375253}}.

\bibitem{Zabczyk08}
J.~Zabczyk.
\newblock {\em Mathematical Control Theory: an Introduction}.
\newblock Modern Birk\''auser Classics. Birkh\"auser, Boston, 2008.
\newblock Reprint of 1995 edition.
\newblock \href {http://dx.doi.org/10.1007/978-0-8176-4733-9}
  {\path{doi:10.1007/978-0-8176-4733-9}}.

\end{thebibliography}

}

\end{document}